\documentclass[12pt]{amsart}

\usepackage[left=2.2cm,tmargin=2.5cm,bmargin=2.5cm,right=2.2cm]{geometry}

\usepackage{physics}
\usepackage[capitalise]{cleveref}
\usepackage{amsmath, amssymb}
\usepackage{graphicx}

\theoremstyle{plain}
\newtheorem{thm}{Theorem}[section]
\newtheorem{lem}[thm]{Lemma}
\newtheorem{cor}[thm]{Corollary}
\newtheorem{prop}[thm]{Proposition}
\newtheorem{pred}[thm]{Prediction}

\begin{document}

\title{The size of wild Kloosterman sums in number fields and function fields}
\author{Will Sawin}

\maketitle

\begin{abstract} We study $p$-adic hyper-Kloosterman sums, a generalization of the Kloosterman sum with a parameter $k$ that recovers the classical Kloosterman sum when $k=2$, over general $p$-adic rings and even equal characteristic local rings. These can be evaluated by a simple stationary phase estimate when $k$ is not divisible by $p$, giving an essentially sharp bound for their size. We give a more complicated stationary phase estimate to evaluate them in the case when $k$ is divisible by $p$. This gives both an upper bound and a lower bound showing the upper bound is essentially sharp. This generalizes previously known bounds \cite{clz} in the case of $\mathbb Z_p$. The lower bounds in the equal characteristic case have two applications to function field number theory, showing that certain short interval sums and certain moments of Dirichlet $L$-functions do not, as one might hope, admit square-root cancellation. \end{abstract}

\section{Introduction}

Let $R$ be a discrete valuation ring of prime residue characteristic $p$, $\pi$ a uniformizer, $n$ and $k$ positive integers, and $\psi$ a nondegenerate character $R/\pi^nR \to \mathbb C^\times$. Fix $k\geq 1$ and define the Kloosterman sum 
\[Kl_k(x)= \sum_{\substack{ x_1,\dots,x_k  \in R/\pi^nR \\ \prod_{i=1}^k x_i =x }} \psi \Bigl( \sum_{i=1}^k x_i\Bigr) .\]

The goal of this paper is to evaluate this sum (including determining when it is zero and bounding it) in the case where $n>1$. In particular, we will handle the trickier case where $p$ divides $k$. This problem is most classical over $R = \mathbb Z_p$, but we will work with both more general $p$-adic rings and rings of equal characteristic $p$ in the interests of applications to function fields, potential future applications to number fields, and the desirability of putting results in their proper, most general context.

We begin by describing the obtained bounds. This requires introducing some notation:

Let $v$ be the $p$-adic valuation of $k$. In mixed characteristic, let $e$ be the $\pi$-adic valuation of $p$. Let
 \begin{equation}\label{w-def} w = \begin{cases} \# \{ j \mid 0\leq j \leq v-1,  p^j(p-1) \mid e,  e (v-j + (p^j+1) / (p^{j+1}-p^j)) \leq  n-1 \}  & \textrm{(mixed characteristic)} \\ 0 & \textrm{(equal characteristic)} \end{cases}\end{equation} and 
\[ k^*= \gcd(k , \abs{R/\pi}-1) p^w .\]     Note that $w\leq v$ and $\gcd(k , \abs{R/\pi}-1)  \leq k/p^v$ so we always have $k^* \leq k$.

  We always take $0\in \mathbb N$. Let  \begin{equation}\label{c-def} c= \min \{ s \in \mathbb N \mid   \pi^{ (p^r+ 1) s } p ^{ v-r} \equiv 0 \bmod \pi^{n}   \textrm{ for all } r \in \mathbb N, r\leq v \}\end{equation} and  \begin{equation}\label{c-tilde-def}  \tilde{c} = \min \{ s \in \mathbb N \mid   \pi^{ (p^r+ 1) s } p ^{ v-r} \equiv 0 \bmod \pi^{n-1}   \textrm{ for all } r \in \mathbb N, r\leq v \}.\end{equation}
  

 The main results of this paper are the upper bound Theorem \ref{main-theorem} and the lower bound Proposition \ref{lower-bound-intro} showing that Theorem \ref{main-theorem} is close to sharp.
\begin{thm}[Propositions \ref{final-even-2} and \ref{final-odd-2}]\label{main-theorem}  If $n\geq 2$, we have
\[ \abs{Kl_k(x)} \leq k^*   \abs{R/\pi}^{k n/2 - c/2 - \tilde{c} /2 } \] \end{thm}
\begin{prop}[Proposition \ref{lower-bound}]\label{lower-bound-intro} If $n \geq 2$, there exists $x \in R/\pi^n$ such that
\[ \abs{Kl_k(x)} \geq  \abs{R/\pi}^{k n/2 - c/2 - \tilde{c}/2 } \] \end{prop}

The estimate of \ref{main-theorem} simplifies in two cases.

\begin{cor} If $n \geq 2$ and $e=1$ we have \[
\abs{Kl_k(x)} \leq \gcd(p,2) \gcd(k, \abs{R/\pi}-1 )   \abs{R/\pi}^{k n/2 -\max \left(  \frac{n-v}{2}  ,1\right)}   \] where $\gcd(p,2)$ is $1$ if $p\neq 2$ and $2$ if $p=2$.
\end{cor}

When $R= \mathbb Z_p$ so $\abs{R/\pi} =p$, this estimate was obtained earlier in \cite{clz}.

\begin{proof} Since $e=1$, we never have $p^j(p-1) \mid e,$ unless $p=2$ and $j=0$, so $k^* = \gcd(k,  \abs{R/\pi}-1)$, except in the $p=2$ case where there is an extra factor of $2$. Furthermore, we have $\tilde{c} = \max \left( \left \lceil \frac{n-1-v}{2}\right \rceil  ,1\right)$ and $c= \max \left( \left \lceil \frac{n-v}{2}\right \rceil  ,1\right)$ so that $c + \tilde{c} = \max( n-v,2)$. \end{proof}

\begin{cor} If $n\geq 2$ and $R$ is a ring of equal characteristic, \[ \abs{Kl_k(x)} \leq k^* \abs{R/\pi}^{ \frac{kn  -  \left \lceil \frac{n}{p^v+1} \right \rceil - \left \lceil \frac{n-1}{p^v+1} \right \rceil}{2} }.\] \end{cor}

Note that this upper bound is roughly of size $ \abs{R/\pi} ^{ \left( \frac{k}{2} - \frac{1}{p^v+1} \right) n}$ and thus is worse than square-root cancellation, which would be an exponent of $\left( \frac{k}{2} - \frac{1}{2} \right) n$.

\begin{proof} We have  $c = \left \lceil \frac{n}{p^v+1} \right \rceil $ and $\tilde{c} = \left \lceil \frac{n-1}{p^v+1} \right \rceil $. \end{proof} 

In the general mixed characteristic case, the situation is more complicated than either of these. We have 
\[ c = \min \{ s \in \mathbb N \mid    (p^r+ 1) s  + e (v-r) \geq n  \textrm{ for all } r \in \mathbb N, r\leq v \} = \max_ {r \in \{0,\dots, v\}} \left\lceil \frac{ n - e (v-r)}{ p^r+1 }\right \rceil .\]
Depending on $n,e,v$, the maximum can be attained at any value of $r$, so there are many regimes where the growth rate of $ \sup_x \abs{Kl_k(x)}$ in $n$ takes different values.

These estimates have interesting consequences for moments of $L$-functions in the function field case. Let $\mathbb F_q$ be a finite field of characteristic $p$, $\mathbb F_q[T]$ the ring of polynomials in one variable over $\mathbb F_q$, $\pi$ a prime polynomial in $\mathbb F_q[T]$, $\mathbb F_q[T]^+_{\pi'} $ the set of monic polynomials in $\mathbb F_q[T]$ prime to $\pi$,  and $n$ a natural number. For $f$ a polynomial write $\abs{f}= q^{\deg f}$.  For $\chi$ a nontrivial Dirichlet character $(\mathbb F_q[T] / \pi^n)^\times \to \mathbb C^\times$, we can define \[ L(s,\chi) =  \sum_{\substack{ f \in \mathbb F_q[T]^+_{\pi'} }} \chi(f)  \abs{f}^{-s} .\] We  say $\chi$ is primitive if it does not factor through $(\mathbb F_q[T] / \pi^{n-1})^\times$ and we say $\chi$ is odd if $\chi(\mathbb F_q^\times)\neq 1$. We let $ \mathcal F_{\pi,n}$ be the set of primitive odd Dirichlet characters mod $\pi^n$. We can consider moments of $L$-functions such as
\[ \sum_{ \chi \in  \mathcal F_{\pi,n}} \abs{ L(1/2,\chi)}^{2k} \] for a natural number $k$ or more general shifted twisted moments such as 
\begin{equation}\label{general-L-moment} \sum_{ \chi \in  \mathcal F_{\pi,n}} \chi(a) \prod_{i=1}^k L(1/2 + \alpha_i, \chi) \overline{L(1/2+ \alpha_{k+i},\chi)}  \end{equation}
for a natural number $k$, shifts $\alpha_1,\dots,\alpha_{2k} \in i \mathbb R$, and $a \in (\mathbb F_q[T] / \pi^n)^\times$. The CFKRS heuristics \cite{CFKRS} and their function field analogues \cite{AK} can be used to provide predictions for such moments. However, in the case of twisted moments, they have usually been used to produce estimates with error terms that are not uniform in the twist $a$ \cite{btb}, and in fact large secondary terms are known to appear \cite[Theorem 10]{cms}. We remedy this by producing a CFKRS-like estimate that could plausibly have a uniform error term of square-root size, by including multiple main terms. We show that for $k=1$ the error term is in fact of square-root size uniformly in $a$.

However, we use our lower bounds for Kloosterman sums to show that, for $k \geq p^v$, the error term of this estimate cannot have power savings better than $1/(p^v+1)$, in the large $n$, fixed $\pi$ limit (i.e. in the depth aspect).  In particular, when $k \geq p$ one cannot obtain square-root cancellation. We expect that this is a large characteristic phenomenon and cautiously predict that uniform square-root cancellation should hold over function fields for $k<p$ and over the integers for all $k$, in particular because this family of Dirichlet $L$-functions is harmonic (in the sense of \cite{sst}) and there still seems to be no evidence that harmonic families over number fields don't admit square-root cancellation in their moments.

Another lower bound applies to sums of divisor-like functions in short intervals.

For $f$ a monic polynomial over $\mathbb F_q$ of degree $k(n-2)$, let $d_k^{(n-2,\dots, n-2)}(f)$ be the number of $k$-tuples $f_1,\dots, f_k$ of monic polynomials of degree $n-2$ such that $\prod_{i=1}^k f_i =f$, which we think of as either an analogue of the generalized divisor function $d_k(n)$ which counts the number of $k$-tuples of positive integers whose product is $n$, or, more precisely, an analogue with factors of restricted size $\sum_{ n_1,\dots, n_k\in \mathbb N, \prod_{i=1}^k n_i =n} \prod_{i=1}^k \theta ( n_i/N)$ for a smooth weight function $\theta$.  Define $\mathcal I_{f, (k-1)(n-2) -1}$ to be $\{f+g \mid g\in \mathbb F_q[T], \abs{g} <  q^{ (k-1)(n-2)-1} \}$, which we think of as a function field analogue of a short interval. 

A special case of \cite[Theorem 4.5]{short-intervals} is that for any $g$ monic of degree $k(n-2)$ over a finite field $\mathbb F_q$ of characteristic $p$, 
\[\left|  \sum_{f \in \mathcal I_{g, (k-1) (n-2)-1  }} d_k^{(n-2,\dots, n-2)}(f)  - q^{ (k-1) (n-2)-1} \right| \ll 3 (k+2)^{ (k+1) (n-2) +1 } q^{ \frac{p+1}{2p} (k-1) n  } .\]

This is an $\mathbb F_q[T]$-analogue of a power savings estimate for the sum of a divisor-like function (with the size of the divisors restricted by smooth weights, say) in a short interval. It has power savings, which approaches square-root cancellation as $p \to \infty$ for fixed $k$, but not for $p$ fixed. Here square-root cancellation would be an error term of size $q^{ (k-1) n/2}$.

As a consequence of our estimates for Kloosterman sums, we can show that this sum in fact fails to admit square-root cancellation when $k$ is divisible by $p$, and the upper bound is closer than it might appear to being sharp when $k=p$ and $q$ is large. 


\begin{prop}\label{interval-lower-bound-intro} For any integers $k \geq 1$ and $n\geq 2$ and a finite field $\mathbb F_q$ of characteristic $p$, we have \[ \Bigl| \sum_{f \in \mathcal I_{g, (k-1) (n-2)-1  } }d_k^{(n-2,\dots, n-2)}(f)  - q^{(k-1)(n-2)-1} \Bigr| \gg q^{ \left( \frac{k}{2}  - \frac{1}{p^v+1} \right) n }  \]  for at least one $g$ monic of degree $k(n-2)$, with the constant depending only on $q$ and $k$.\end{prop} 

  In the case $k=p$, so $v=1$ and $p^v=1$, this gives an exponent of $\frac{p}{2} - \frac{1}{p+1}$ in $q^n$, which differs from the upper bound $\frac{p+1}{2p} (p-1) = \frac{p}{2} - \frac{1}{2p}$ by $\frac{p-1}{ 2p (p+1)}$. Thus, the difference between the lower and upper bounds is less than the difference between the upper bound and the GRH bound $\frac{p}{2}$. 

I would like to thank Mark Shusterman, Julio Andrade, Jon Keating, and Brian Conrey for several helpful conversations and comments on this manuscript, as well as the anonymous referee for many helpful comments. This research was supported by NSF grant DMS-2101491.

\section{Preliminaries}

We begin with a bound for a general class of Gauss sums.

\begin{lem}\label{general-gauss-bound} Let $\kappa$ be a finite field, $V$ a finite-dimensional vector space over $\kappa$, and \[ \varphi \colon V \to \{ z \in \mathbb C \mid \abs{z} = 1\}\] a function. Let\[\widetilde{\varphi} (v,w) = \varphi(v+w)\overline{ \varphi(v)} \overline{\varphi(w)} \varphi(0 ). \]  Assume that $w \mapsto \widetilde{\varphi} (v,w)$ is a group homomorphism $V \to \mathbb C^\times$ for each $v\in V$.

Let $W$ be the kernel of $\widetilde{\varphi}$, i.e. the set of $v \in V$ with $\widetilde{\varphi} (v,w)=1$ for all $w\in V$. Then
\[ \abs{ \sum_{v\in V} \widetilde{\varphi}(V) }= \begin{cases} \sqrt{ \abs{V} \abs{W}}  & \textrm{if }\varphi\textrm{ is constant on } W \\ 0 & \textrm{otherwise}\end{cases}.\]

Furthermore, in the special case $\varphi(v) = \psi(Q (v))$ for $\psi \colon \mathbb F_q \to \mathbb C^\times$ a nontrivial character and $Q \colon V \to \kappa $ a polynomial of degree $\leq 2$, the set $W$ is a subspace of $V$, the kernel of the biliinear form \[B(v,w) = Q(v+w) - Q(v) - Q(w)+Q(0)\] and thus $\sqrt{\abs{V} \abs{W}} =\abs{\kappa}^{ \frac{\dim V+ \dim W}{2}}$.

\end{lem}

\begin{proof} We have
\[ \abs{ \sum_{v\in V} \varphi(v) }^2 = \sum_{v,w \in V} \varphi (v) \overline{ \varphi(w)}= \sum_{v\in V} \sum_{w\in V}  \varphi(v+w)  \overline{\varphi(w)} = \sum_{v\in V} \varphi(v) \overline{\varphi(0)} \sum_{w\in V} \widetilde{\varphi}( v,w)   .\]
Since $\widetilde{\varphi}(v,\cdot)$ is a group homomorphism,  $\sum_{w\in V} \widetilde{\varphi}( v,w) =0$ unless $\widetilde{\varphi}(v,\cdot)$ is trivial, i.e. $v\in W$, and equals $\abs{V}$ if $v\in W$. Thus
\[ \abs{ \sum_{v\in V} \varphi(v) }^2 = \abs{V} \sum_{v\in W} \varphi(v) \overline{\varphi(0)}  .\]

Since $\widetilde{\varphi}$ is symmetric, $v\mapsto \widetilde{\varphi}(v,w)$ is a group homomorphism for each $w$, and since $W$ is the intersection of the kernels of all these group homomorphisms, it is also a finite group. For $v,w\in W$, we have
\[  \varphi(v+w) \overline{\varphi(0)} = \varphi(v) \overline{\varphi(0)}  \varphi(w) \overline{\varphi(0)} \widetilde{\varphi}(v,w) = \varphi(v) \overline{\varphi(0)}  \varphi(w) \overline{\varphi(0)}\] so $v\mapsto  \varphi(v) \overline{\varphi(0)}$ is a group homomorphism. Thus $\sum_{v\in W} \varphi(v) \overline{\varphi(0)} $ vanishes unless $v\mapsto  \varphi(v) \overline{\varphi(0)}$ is trivial, in which case it is $\abs{W}$, giving \[ \abs{ \sum_{v\in V} \varphi(v) }^2 = \abs{V} \abs{W}.\]

This gives the statement since $v\mapsto  \varphi(v) \overline{\varphi(0)}$ is trivial on $W$ if and only if $\varphi$ is constant on $W$.

In the quadratic polynomial case, we have $\widetilde{\varphi}(v,w) = \psi(B (v,w))$, and, since every nonzero linear form is surjective and thus nonconstant when composed with $\psi$,  we have $v\in W$ if and only if $v$ is in the kernel of $B$.
\end{proof}

The next few lemmas are devoted to finding the largest $\pi$-adic intervals on which the function $ \psi\left ( (k-1) a + \frac{x}{a^{k-1} } \right) $, which we will sum in \eqref{eq-first-phase-even}, behaves like an additive character, so that we can obtain cancellation in the sums when the character is nontrivial. We begin with a lemma on the $p$-adic valuation of multinomial coefficients.
\begin{lem}\label{multinom-bound} For any $i_1,i_2>0$, there exists some $r \geq 0$ such that   \begin{equation}\label{first-ineq} i_1 + i_2 \geq p^r+1\end{equation} and \begin{equation}\label{second-ineq} v_p \left(  \binom{k+i_1+ i_2-2}{ i_1, i_2 , k-2 }   \right) \geq v - r .\end{equation}
Furthermore, we can choose $r$ so that one of these inequalities is strict,  $(i_1,i_2)= (p^r,1)$, or $(i_1,i_2)=(1,p^r)$.\end{lem}
\begin{proof} Choose $r$ to be maximal such that $i_1+ i_2 \geq p^{r}+1$, so in particular $i_1 + i_2 \leq p^{r+1}$ and hence $i_1,i_2 < p^{r+1}$. Then $v_p \bigl(  \binom{k+i_1+ i_2-2}{ i_1, i_2 , k-2 }   \bigr) $ is the number of carries when adding $k-2$, $i_1$, and $i_2$ together in base $p$ \cite[Theorem 7]{Singmaster}.  For the first part, it suffices to check there is a carry in every place from $r+1$ to $v$.

There is a carry in the $d$th place if and only if  we have\[ i_1 \bmod p^{d} + i_2 \bmod p^d + (k -2 ) \bmod p^d  >( k+i_1 + i_2 - d )\bmod p^d\] where $\textrm{mod } p^d$ is understood to be the operation that gives the unique representative of each residue class between $0$ and $p^d-1$.  Fix any $d$ with $r+1 \leq d \leq v$, so in particular that $p^d \mid k$. Since $i_1, i_2 < p^{r+1} < p^d$, we have $i_1 \bmod p^d= i_1$ and $i_2 \bmod p^{d}= i_2$. Thus
\[ i_1 \bmod p^{d} + i_2 \bmod p^d + (k -2 ) \bmod p^d   \geq i_1 + i_2 +  p^d-2 \geq 1+ 1+ p^d-2\] \[=p^d > ( k+i_1 + i_2 - d )\bmod p^d\]
so indeed there is a carry in the $d$th place, as desired.

If \eqref{first-ineq} is not strict, then $i_1+i_2= p^r+1$. Unless one of $i_1,i_2$ is equal to $1$, this implies there is a carry when adding $i_1$ to $i_2$ in some place from $0$ to $r-1$, which means that \eqref{second-ineq} is strict.\end{proof}

\begin{lem}\label{hom-to-multinom} For any $a \in R^\times$ and $y_1,y_2\in \pi R$, we have
\[  (a+y_1+y_2) ^{1-k} - (a+y_1)^{1-k} - (a+y_2)^{1-k} + a^{1-k} \] \[ = \sum_{i_1,i_2=1}^{\infty} (-1)^{i_1+i_2}  \binom{k + i_1 + i_2 -2}{i_1,i_2, k-2} y_1^{i_1} y_2^{i_2} a^{ 1-k-i_1-i_2}. \]\end{lem}

\begin{proof} The Taylor series for  $(1+y/a)^{-1}$ gives
\[  (a+y_1+y_2) ^{1-k}  = \sum_{i_1,i_2=0}^{\infty} (-1)^{i_1+i_2}  \binom{k + i_1 + i_2 -2}{i_1,i_2, k-2} y_1^{i_1} y_2^{i_2} a^{ 1-k-i_1-i_2} \] and the result follows by cancelling terms. These series converge $\pi$-adically since the binomial coefficients are integers while $y_1^{i_1} y_2^{i_2}$ is divisible by $\pi^{i_1+i_2}$ and so there are only finitely many terms not divisible by a given power of $\pi$. \end{proof} 

\begin{lem}\label{homomorphism-equation} Recall $c$ from \eqref{c-def}. For $a \in R^\times$ and  $y_1, y_2 \in \pi^c R$ we have 
\[  (a+y_1+y_2) ^{1-k} - (a+y_1)^{1-k} - (a+y_2)^{1-k} + a^{1-k} \equiv 0 \bmod \pi^n .\]
\end{lem} 

\begin{proof} By \cref{hom-to-multinom}, it suffices to prove that $\binom{k + i_1 + i_2 -2}{i_1,i_2, k-2} y_1^{i_1} y_2^{i_2}$ is divisible by $\pi^n$ for all $i_1, i_2$. Fix some $i_1,i_2\geq 1$. By \cref{multinom-bound}, there exists $r$ such that $i_1 + i_2 \leq p^r+1$ and $\binom{k + i_1 + i_2 -2}{i_1,i_2, k-2} $ is divisible by $p^{v-r}$, so  $\binom{k + i_1 + i_2 -2}{i_1,i_2, k-2} y_1^{i_1} y_2^{i_2}$ is divisible by $p^{v-r}   \pi^{ (p^r+1) c}$ and hence is divisible by $\pi^n$ by \eqref{c-def}. \end{proof}

Let $\mathcal S$ be the set of  $(a,x) \in (R/\pi^n R)^2$ such that
 \begin{equation}\label{key-equation}   \psi \left( y( k-1)  + \frac{x}{ (a+y)^{k-1} }- \frac{x}{ a^{k-1} } \right) = 1 \textrm{ for all } y\in \pi^c R \end{equation}  Lemmas \ref{second-phase-even} and \ref{second-phase-odd} will express $Kl_k(x)$ as a sum over $a$ with  $(a,x) \in \mathcal S$, so understanding $\mathcal S$ will be important. We begin with a couple of preparatory lemmas.

\begin{lem}\label{E-to-power} For $n$ even, if $(a,x)\in \mathcal S$ then $a^k \equiv x \bmod \pi^{n/2}$. \end{lem}

\begin{proof} For $y$ divisible by $\pi^{n/2}$ (and thus automatically divisible by $\pi^c$ since $\pi^n \mid \pi^{ (p^r+ 1) n/2 } \mid \pi^{ (p^r+ 1) n/2 } p ^{ v-r}$ for all $r\in \mathbb N, r\leq v$), using $O(y^2)$ to denote an $R$-multiple of $y^2$, we have
\[y ( k-1)  + \frac{x}{ (a+y)^{k-1} } - \frac{x}{ a^{k-1} } =  y(k-1) + \frac{x}{ a^{k-1} } + \frac{x y (1-k)}{ a^k }+ O(y^2) - \frac{x}{ a^{k-1} } \] \[ \equiv y(k-1) + \frac{x  y (1-k)}{ a^k } = y (k-1)  \left( 1- \frac{x}{a^k} \right)  \bmod \pi^{n} \] and supposing for contradiction that $x /a^k \not\equiv 1 \bmod \pi^{n/2}$, we have  $(k-1)  \left( 1- \frac{x}{a^k}\right) \not\equiv 0\bmod \pi^{n}$, so because $\psi$ is nondegenerate, we can always find $y$ where $\psi \left(  y (k-1)  \left( 1- \frac{x}{a^k} \right) \right)\neq 1$, contradicting \eqref{key-equation}. \end{proof}

\begin{lem}\label{kth-power-depends} For any $a \in R/\pi^n R$, the congruence class of $a^k \bmod \pi^{\lceil \frac{n}{2} \rceil}$ depends only on the congruence class of $a$ modulo $\pi^c$.

If $v> 0$, it furthermore only depends on the congruence class of $a$ modulo $\pi^{\tilde{c}}$, recalling $\tilde{c}$ from \eqref{c-tilde-def}.\end{lem}

\begin{proof} Indeed, for $z\in \pi^c R$,  $(a+z)^k - a^k = \sum_{i=1}^k \binom{k}{i} z^i a^{k-i}$ and for $p^r \leq i < p^{r+1}$ we have $v_p \bigl( \binom{k}{i} \bigr) \geq v-r$. Using \eqref{c-def} and the fact that $p^r \geq 1$ we have
\[ \pi^n \mid p^{v-r} \pi^{(p^r+1)c} \mid (p^{v-r} \pi^{p^r c})^2 \]
which implies
\[ \pi^{\lceil n/2\rceil} \mid p^{v-r} \pi^{p^r c} \mid  \binom{k}{i} z^i a^{k-i}.\]

Because this holds for all $i$, we have  $\pi^{\lceil n/2\rceil} \mid (a+z)^k - a^k $, so $(a+z)^k$ and $a^k$ share the same congruence class.

Substituting $\tilde{c}$ for $c$ in this argument, the only change is that $p^{v-r} \pi^{(p^r+1)\tilde{c}}$ may be divisible only by $\pi^{n-1}$. To obtain the same conclusion, it thus suffices to check that $(p^{v-r} \pi^{p^r c})^2$ is divisible by $p^{v-r} \pi^{(p^r+1)c+1}$. This is true as long as $v>r$ or $p^r>1$. If $v>0$, one of these two cases always occurs.  \end{proof}

\begin{lem}\label{a-shift-criterion} Let $a,x,z\in R/\pi^nR$. Suppose $(a,x)\in \mathcal S$. Then $(a+z, x) \in \mathcal S$ if and only if 
\[ \psi \Biggl( \sum_{i_1,i_2=1}^\infty(-1)^{i_1+i_2}  \binom{k + i_1 + i_2 -2}{i_1,i_2, k-2} y^{i_1} z ^{i_2} a^{ 1-k-i_1-i_2} \Biggr) =1\textrm{ for all } y\in \pi^c R.\] \end{lem}

\begin{proof} By definition, $(a+z, x)\in \mathcal S$ if and only if 
  \[  \psi \left( y( k-1)  + \frac{x}{ (a+z+y)^{k-1} }- \frac{x}{ (a+z)^{k-1} } \right) = 1 \textrm{ for all } y\in \pi^c R\]
 which by \eqref{key-equation} for $(a,x)$ occurs if and only if
\[  \psi \left(  \frac{x}{ (a+z+y)^{k-1} }- \frac{x}{ (a+z)^{k-1} }  - \frac{x}{ (a+y)^{k-1} }+ \frac{x}{ a^{k-1} } \right) = 1 \textrm{ for all } y\in \pi^c R\]
and by \cref{hom-to-multinom}, the term inside the $\psi$ is
\begin{equation}\label{tricky-sum}  \sum_{i_1,i_2=1}^\infty(-1)^{i_1+i_2}  \binom{k + i_1 + i_2 -2}{i_1,i_2, k-2} y^{i_1} z ^{i_2} a^{ 1-k-i_1-i_2} .\end{equation}
\end{proof}

Studying the sum \eqref{tricky-sum} will be crucial to the next few lemmas.

 \begin{lem}\label{a-uniqueness} Whether or not $(a,x)\in \mathcal S$ depends only on $a$ modulo $\pi^{\tilde{c}}$.\end{lem}
 
 \begin{proof}Let $z \in \pi^{\tilde{c}}R$. By \cref{a-shift-criterion}, it suffices to check for each  $y\in \pi^{c} R$ and each $i_1,i_2>0$ that $ \binom{k + i_1 + i_2 -2}{i_1,i_2, k-2} y^{i_1} z ^{i_2} $ is divisible by $\pi^n$. By \cref{multinom-bound}, there exists $r$ with $i_1 + i_2 \geq p^r+1$ and $ \binom{k + i_1 + i_2 -2}{i_1,i_2, k-2}$ divisible by $p^{v-r}$. 

Noting that $c\geq \tilde{c}$ by definition, if $c =\tilde{c}$ then $\binom{k + i_1 + i_2 -2}{i_1,i_2, k-2} y^{i_1} z ^{i_2} $ is divisible by $ p^{v-r} \pi^{ (p^r+1) c}$ and thus by \eqref{c-def} is divisible by $\pi^n$, and if $c>\tilde{c}$ then $c \geq \tilde{c}+1$ so $\binom{k + i_1 + i_2 -2}{i_1,i_2, k-2} y^{i_1} z ^{i_2} $ is divisible by $ p^{v-r} \pi^{ (p^r+1) \tilde{c} +1 }$ and thus by \eqref{c-tilde-def} is divisible by $\pi^n$. \end{proof}

\begin{lem}\label{a-shift-valuation} Let $a,x,z \in R$. Let $u$ be the $\pi$-adic valuation of $z$.

Suppose either (i) that $(a,x)\in \mathcal S $,  $(a+z,x) \in \mathcal S$, and $0<u< \tilde{c}$ or (ii) that $u=\tilde{c}<c $ and \[ \psi \Biggl( \sum_{i_1,i_2=1}^\infty(-1)^{i_1+i_2}  \binom{k + i_1 + i_2 -2}{i_1,i_2, k-2} y^{i_1} z ^{i_2} a^{ 1-k-i_1-i_2} \Biggr) =1\textrm{ for all } y\in \pi^{\tilde{c}}  R.\]

Then $R$ is a ring of mixed characteristic and $u = \frac{e}{p^{j+1}- p^j}$ for some $j$ from $0$ to $v-1$. Furthermore for each $a,x,j$, there are at most $p$ possible values of $ z$ modulo $\pi^{u+1}$.\end{lem}

\begin{proof} Choose $j \in \{0,\dots, v\}$ minimizing the $\pi$-adic valuation of $ z^{p^j} p^{ v-j}$. In particular, in a ring of equal characteristic $p$, we have $j=v$, and in a ring of mixed characteristic, we have $u \geq  \frac{e}{p^{j+1} - p^j }$ unless $j=v$ and $ u\leq \frac{e}{ p^j - p^{j-1}} $ unless $j=0$.

Let $y$ have $\pi$-adic valuation $n-1 - e(v-j) - u p^j$, so that  $ y z^{p^j} p^{v-j}$ has $\pi$-adic valuation $n-1$. (Here $e(v-j)$ is taken to be $0$ if $R$ has equal characteristic and thus $v=j$, even though $e$ is undefined in this case.)

Then in case (ii), we can check that $y$ is divisible by $\pi^{\tilde{c}}$. Since $u=\tilde{c} <c$, we must have $z^{ (p^r+ 1) } p ^{ v-r} $ not divisible by $\pi^{n}$ for some $r$, so $z^{p^r} \pi^{\tilde{c}} p^{v-r}$ not divisible by $\pi^n$ and thus $z^{p^j} \pi^{\tilde{c}} p^{v-j}$ is not divisible by $\pi^n$, so $y$ is divisible by $\pi^{\tilde{c}}$.

Similarly, in case (i), we can check that $y$ is divisible by $\pi^c$.  By \eqref{c-tilde-def}, we have $  \pi^{ (p^r+1) (\tilde{c}-1)} p^{v-r}$ not divisible by $\pi^{n-1}$ for some $r$, so we have $ z^{p^r} \pi^{\tilde{c} -1 } p^{v-r}$ not divisible by $\pi^{n-1}$ for some $r$, so $z ^{p^j} \pi^{\tilde c-1} p^{v-j}$ is not divisible by $\pi^{n-1}$, so $y$ is divisible by $\pi^{\tilde{c}}$. This gives the claim unless $c>\tilde{c}$, in which case $u \leq c-2$ and by \eqref{c-def}, we have  $ \pi^{ (p^r+1) (c-1)} p^{v-r}$ not divisible by $\pi^{n}$ for some $r$, so we have $ z^{p^r} \pi^{c -1 } p^{v-r}$ not divisible by $\pi^{n-p^r}$ and in particular not divisible by $\pi^{n-1}$, so $z^{p^j} \pi^{c-1} p^{v-j}$ is not divisible by $\pi^{n-1}$, so $y$ is divisible by $\pi^c$.

In either case, it follows that \[ \psi \Biggl( \sum_{i_1,i_2=1}^\infty(-1)^{i_1+i_2}  \binom{k + i_1 + i_2 -2}{i_1,i_2, k-2} y^{i_1} z ^{i_2} a^{ 1-k-i_1-i_2} \Biggr) =1,\] using \cref{a-shift-criterion} in case (i).

Now we will show that almost all the terms in the sum \eqref{tricky-sum} are divisible by $\pi^{n}$.

Indeed, given $i_1,i_2$, by \cref{multinom-bound} we may choose $r$ so that $i_1 + i_2 \geq p^r+1$ and $ \binom{k + i_1 + i_2 -2}{i_1,i_2, k-2}$ is divisible by $p^{v-r}$. Since the $\pi$-adic valuation of $z$ is less than the $\pi$-adic valuation of $y$,  unless $i_1=1$,
\[ \pi^n \mid \pi p^{v-j} y z^{p^j} \mid \pi p^{v-r} y z^{ p^r}\mid \pi \binom{k + i_1 + i_2 -2}{i_1,i_2, k-2} y^{1} z ^{i_1+i_2-1}\mid  \binom{k + i_1 + i_2 -2}{i_1,i_2, k-2} y^{i_1} z ^{i_2} .\]
 
 Even if $i_1=1$, a similar reasoning works unless $i_2= p^r$. If $i_2 =p^r$, the $p$-adic valuation of $\binom{k + i_1 + i_2 -2}{i_1,i_2, k-2}=  \binom{k + 1+ p^r-2}{1,p^r, k-2} $ is exactly $r$, so $ \binom{k + i_1 + i_2 -2}{i_1,i_2, k-2} y^{i_1} z ^{i_2} = \binom{k + 1 + p^r -2}{1,p^r, k-2} y^{1} z ^{p^r} $ has $\pi$-adic valuation exactly \begin{equation}\label{valuation-at-least-n-1}  e (v-r) + p^r u + (n-1 - e (v-j) - p^j u) \geq n-1\end{equation} by the definition of $j$.
 
 Equality in \eqref{valuation-at-least-n-1} holds if and only if \[ e (v-r) +p^r u =e (v-j) +  p^j u.\]  In particular, it holds for $r=j$, and because $e (v-j) +  p^j u$ is a strictly convex function of $j$, for at most one other value of $j$: for $r=j-1$ if $u = \frac{e }{p^j - p^{j-1}}$ and for $r = j+1$ if $u = \frac{e}{p^{j+1}- p^j}$. 

If equality in \eqref{valuation-at-least-n-1} does not hold for any $j \neq r$, then $(a+z,x)\notin \mathcal S$. Indeed, the sum in \eqref{tricky-sum} contains exactly one term which is nonvanishing mod $\pi^n$,  \[(-1)^{p^j +1} \binom{ k + 1+ p^j-2}{1,p^j,k-2} yz^{p^j} a^{-k-p^j},\] and this term has $\pi$-adic valuation $n-1$. Thus, multiplying $y$ by a unit, we can make this term, and thus \eqref{tricky-sum}, be an arbitrarily element of $\pi^{n-1} (R/\pi)^\times$. Choosing the unit appropriately, we can make $\psi$ nontrivial on \eqref{tricky-sum}.

On the other hand, if equality in \eqref{valuation-at-least-n-1} holds for some $j \neq r$, then possibly after switching $r$ and $j$, we have $r = j+1$ and $u = \frac{e}{p^{j+1}- p^j}$. (In particular, this is never satisfied if $R$ has equal characteristic and thus $e=\infty$.) In this case, \eqref{tricky-sum} contains exactly two terms which are nonvanishing mod $\pi^{n}$ and thus is congruent mod $\pi^n$ to 
\[ (-1)^{p^j +1} \binom{ k + 1+ p^j-2}{1,p^j,k-2} yz^{p^j} a^{-k-p^j} + (-1)^{p^{j+1}  +1} \binom{ k + 1+ p^{j+1} -2}{1,p^{j+1} ,k-2} yz^{p^{j+1} } a^{-k-p^{j+1} }.\]
Note that both terms have $\pi$-adic valuation $n-1$ by assumption. If their sum  has $\pi$-adic valuation $n-1$, then $(a+z,x)\notin \mathcal S$ for the same reason. So $(a+z,x)\in \mathcal S$ only if \[  (-1)^{p^j +1} \binom{ k + 1+ p^j-2}{1,p^j,k-2} yz^{p^j} a^{-k-p^j} + (-1)^{p^{j+1}  +1} \binom{ k + 1+ p^{j+1} -2}{1,p^{j+1} ,k-2} yz^{p^{j+1} } a^{-k-p^{j+1} } \] \[\equiv 0 \bmod \pi^n. \]
This condition depends only on $z \bmod \pi^{u+1}$, and hence can be viewed as an equation in $R/\pi$ satisfied by $z/\pi^u$. This equation has the form $ \alpha (z/\pi^u)^{p^j} + \beta (z/\pi^u)^{p^{j+1}}\equiv 0 \bmod \pi$ for $\alpha, \beta \in (R/\pi)^\times$, and thus has at most $p$ solutions. \end{proof}

Define $k'= \gcd(k , \abs{R/\pi}-1) p^{ w'} $ where
\begin{equation}\label{w'-def} w' = \begin{cases} \# \{ j \mid 0\leq j \leq v-1,  p^j(p-1) \mid e,  e (v-j + (p^j+1) / (p^{j+1}-p^j)) <  n-1 \}  & \textrm{(mixed characteristic)} \\ 0 & \textrm{(equal characteristic)}.\end{cases} \end{equation}
\eqref{w'-def} differs from the definition \eqref{w-def} of $w$ only in including the strict inequality $<n-1$ instead of $\leq n-1$, so that $w'\leq w$ and thus  $k' \leq k^*$.

\begin{lem}\label{a-counting}  For $x\in (R/\pi^nR)$, the number of congruence classes $a \bmod \pi^{\tilde{c}}$ with $(a,x)\in \mathcal S$ is at most $k'$.  \end{lem}
\begin{proof} 

First note that if $(a_1,x)$ and $(a_2,x)$ both lie in $\mathcal S$ then by \cref{E-to-power}, $a_1^k \equiv x = a_2^k \bmod \pi^{n/2}$ and so $a_1^k \equiv a_2^k \equiv x \bmod \pi$.

For each $x$, there are at most $\gcd(k , \abs{R/\pi}-1)$ congruence classes modulo $\pi$ satisfying this equation, and thus at most $\gcd(k , \abs{R/\pi}-1)$ congruence classes mod $\pi$ containing $a$ with $(a,x)\in \mathcal S$.

If $R$ has equal characteristic $p$, then two $a$ with $(a,x)\in \mathcal S$ that are congruent mod $\pi$ are congruent mod $\pi^{ \tilde{c}}$ by \cref{a-shift-valuation}(i), so there are at most $\gcd(k , \abs{R/\pi}-1)$ congruence classes mod $\pi^{\tilde{c}}$ containing $a$ with $(a,x)\in \mathcal S$, as desired.

If $R$ has mixed characteristic $p$, then for $0<d < \tilde{c}-1$, by \cref{a-shift-valuation}(i) two $a$ with $(a,x)\in \mathcal S$ that are congruent mod $\pi^d$ are congruent modulo $\pi^{d+1}$, unless $d= e / (p^{j+1} -p^j)$ for some $j$ from $0$ to $v-1$. For each special value of $d$, there are at most $p$ congruence classes modulo $\pi^{d+1}$ containing such $a$ in each congruence class modulo $\pi^d$.  Thus, by induction on $d$, the number of such $a$ modulo $\pi^{d+1}$ is  \[\gcd(k , \abs{R/\pi}-1) p^{ \# \{ j \mid 0\leq j \leq v-1,  p^j(p-1) \mid e,  e/ (p^{j+1} - p^j) \leq d \} } .\] and so the number of such $a$ modulo $\pi^{ \tilde{c}} $ is \[\gcd(k , \abs{R/\pi}-1) p^{ \# \{ j \mid 0\leq j \leq v-1,  p^j(p-1) \mid e,  e/ (p^{j+1} - p^j) < \tilde{c}  \} } .\]
By \eqref{c-tilde-def},  if $e/ (p^{j+1} - p^j)  <  \tilde{c}$ then $ (p^j+1)  e/ (p^{j+1} - p^j)  + e (v-j) < n-1$, so the number of such $a$ is at most 
\[ \gcd(k , \abs{R/\pi}-1) p^{ \# \{ j \mid 0\leq j \leq v-1,  p^j(p-1) \mid e,  e (v-j + (p^j+1) / (p^{j+1}-p^j)) < n-1 \} } = \gcd(k , \abs{R/\pi}-1) p^{w' }= k' .\]
\end{proof}

  \begin{lem}\label{x-uniqueness} For $(a,x) \in (R/\pi^n R)^2$, whether or not  $(a,x) \in \mathcal S$ depends only on $x$ modulo $\pi^{n-c}$. 

For each $a\in (R/\pi^n R) $, there exists a unique congruence class of $x$ mod $\pi^{n-c}$ with $(a,x)\in \mathcal S$. \end{lem}

\begin{proof} There are three claims: depending only on $x$ modulo $\pi^{n-c}$, existence, and uniqueness.

To show it depends only on $x$ mod $\pi^{n-c}$, we note simply that $\frac{1}{ (a+y)^{k-1} }- \frac{1}{ a^{k-1} }$ is divisible by $y$, thus divisible by $\pi^c$, so $ \frac{x}{ (a+y)^{k-1} }- \frac{x}{ a^{k-1} }$ modulo $\pi^n$ depends only on $x$ modulo $\pi^{n-c}$.

For uniqueness, suppose $(a,x)$ and $(a,x+z)$ both lie in $\mathcal S$, where $z$ is not divisible by $\pi^{n-c} $. Then dividing \eqref{key-equation} for $x+z$ by \eqref{key-equation} for $x$, we obtain
\[\psi \left(  \frac{z}{ (a+y)^{k-1} }- \frac{z}{ a^{k-1} } \right) = 1 \textrm{ for all } y\in \pi^c R\]
Taking $y$ of $\pi$-adic valuation $n-1 - v_\pi(z)$, we see that $zy$ has $\pi$-adic valuation $n-1$, and thus, modulo $\pi^n$,
\[  \frac{z}{ (a+y)^{k-1} }- \frac{z}{ a^{k-1} } \equiv zy \left( \frac{1}{ y(a+y)^{k-1} }- \frac{1}{y a^{k-1} } \bmod \pi\right) = zy \left( \frac{1-k}{a^k} \bmod \pi \right).\]
By multiplying $y$ by a suitable element of $(R/\pi)^\times$, we can make $zy \left( \frac{1-k}{a^k} \bmod \pi \right)$ into any element of $\pi^{n-1}(R/\pi)^\times$, and thus we can ensure $\psi$ is nontrivial on it, a contradiction.

For existence, it suffices by induction to show that if $d\geq c$ and $x$ satisfies the equation
\[\psi \left( y(k-1) +  \frac{x}{ (a+y)^{k-1} }- \frac{x}{ a^{k-1} } \right) = 1 \textrm{ for all } y\in \pi^{d+1} R\]
then there exists $x'$ satisfying the same equation for all $y \in \pi^d R$.  Given such an $x$, by \cref{homomorphism-equation}, we see that $\psi \left( y(k-1) +  \frac{x}{ (a+y)^{k-1} }- \frac{x}{ a^{k-1} } \right)$ is a homomorphism from $\pi^d R $ to $\mathbb C^\times$, and since it takes the value $1$ on all $y \in \pi^{d+1}R$, a homomorphism $\pi^d R/ \pi^{d+1} R \to \mathbb C^\times$. Since $\psi$ is nondegenerate, any such homomorphism can be written as $ y\mapsto \psi(zy)$ for some $z$ divisible by $\pi^{n-1-d}$.  Take
\[ x' = x +  \frac{ a^k z}{ k-1} \]
to obtain
\[\psi \left( y(k-1) +  \frac{x'}{ (a+y)^{k-1} }- \frac{x'}{ a^{k-1} } \right) = \psi \left( y(k-1) +  \frac{x}{ (a+y)^{k-1} }+ \frac{ a^k z}{ (k-1) (a+y)^{k-1}} - \frac{x}{ a^{k-1} }  - \frac{ az}{ k-1}\right) \] \[ = \psi \left( y(k-1) +  \frac{x}{ (a+y)^{k-1} }- \frac{x}{ a^{k-1} }  \right) \psi\left(     \frac{ a^k z}{ (k-1) (a+y)^{k-1}}-  \frac{ az}{ k-1}\right)\] \[ = \psi (zy) \psi\left(     \frac{ a^k z}{ (k-1) (a+y)^{k-1}} - \frac{ az}{ k-1} \right) = \psi \left( zy +   \frac{ a^k z}{ (k-1) (a+y)^{k-1}} - \frac{ az}{ k-1}\right)\] \[ = \psi \left(  zy +  \frac{ az}{k-1} -zy +  \frac{ kz y^2}{ 2a}  - \frac{ k (k+1) z y^3}{ 6 a^2} + \dots  - \frac{az}{k-1} \right) = 1 \]
since all the terms that do not cancel are divisible by $zy^2$, hence divisible by $\pi^{ n-1-d+ 2d } = \pi^{n+d-1}$ and thus divisible by $\pi^n$.\end{proof}

\section{Bounds for Kloosterman sums}

We begin with the proof of the upper bound \cref{main-theorem} in the $n$ even case, and then give the proof in the $n$ odd case, which is similar, but slightly more complicated, before finally proving the lower bound (for all $n$).

We begin with a stationary phase analysis that reduces the even case to a one-variable sum.

\begin{lem}\label{first-phase-even} For $n$ even, we have \begin{equation}\label{eq-first-phase-even} Kl_k(x) = \sum_{ a \in R/\pi^n , a^k \equiv x \bmod \pi^{n/2}} \psi\left ( (k-1) a + \frac{x}{a^{k-1} } \right) \abs{R/\pi}^{ (k-2) n/2}.\end{equation}\end{lem}

\begin{proof} Pick a set $S$ of representatives of congruence classes in $R/\pi^{n/2}$. Write each $x_i$ as $a_i+b_i$ where $a_i \in S$ and $b_i$ is divisible by $\pi^{n/2}$.

Then

\[Kl_k(x) =\sum_{ \substack {a_1,\dots, a_k \in S\\ \prod_{i=1}^k a_i \equiv x \bmod \pi^{n/2}}}   \sum_{\substack{ b_1,\dots,b_k   \in \pi^{n/2} R/(\pi^n) \\ \prod_{i=1}^k (a_i+b_i)  =x }} \psi \Bigl( \sum_{i=1}^k a_i+ \sum_{i=1}^k b_i\Bigr ) .\]

Since $b_i b_j = 0 $ for all $i,j$, the equation $ \prod_{i=1}^k (a_i+b_i)  =x $ simplifies to \begin{equation}\label{aff-eq-even} x=   \Bigl( 1+ \sum_{i=1}^k\frac{b_i}{a_i} \Bigr) \prod_{i=1}^k a_i  .\end{equation}  The sum over $b_i$ vanishes unless the character $\psi \Bigl( \sum_{i=1}^k a_i+ \sum_{i=1}^k b_i\Bigr ) $ is constant over the affine hyperplane of solutions $(b_1,\dots, b_k)$ to \eqref{aff-eq-even}, which occurs only if $a_1 = a_2 = \dots =a_k$ since if $a_i \neq a_j$ we can add a multiple of $a_j$ to $b_i$ and subtract a corresponding multiple of $a_i$ from $b_j$ to change the value of the character.

 Say the $a_i$ are equal to $a$. In this case,  \eqref{aff-eq-even} implies that \[\sum_{i=1}^k b_i =  a\left( \frac{ x}{ \prod_{i=1}^k a_i} - 1\right) = \frac{x}{a^{k-1}} -a\] so  \[  \psi \Bigl( \sum_{i=1}^k a_i+ \sum_{i=1}^k b_i\Bigr )  = \psi \left ( (k-1) a + \frac{x}{a^{k-1} } \right) .\]
 
 Furthermore \eqref{aff-eq-even} has exactly  $ \abs{R/\pi}^{ (k-1) n/2}$ solutions since $b_k$ is uniquely determined by $b_1,\dots, b_{k-1}$. Thus
\[Kl_k(x) =\sum_{ \substack {a_1,\dots, a_k \in S\\ \prod_{i=1}^k a_i \equiv x \bmod \pi^{n/2}}}   \sum_{\substack{ b_1,\dots,b_k   \in \pi^{n/2} R/(\pi^n) \\ \prod_{i=1}^k (a_i+b_i)  =x }} \psi \Bigl( \sum_{i=1}^k a_i+ \sum_{i=1}^k b_i\Bigr )\] \[ = \sum_{ \substack {a\in S\\ a^k \equiv x \bmod \pi^{n/2}}}   \sum_{\substack{ b_1,\dots,b_k   \in \pi^{n/2} R/(\pi^n) \\ \prod_{i=1}^k (a+b_i)  =x }}  \psi \left ( (k-1) a + \frac{x}{a^{k-1} } \right)  \] \[= \sum_{ \substack{ a \in S ,\\a^k \equiv x \bmod \pi^{n/2}} }\psi \left ( (k-1) a + \frac{x}{a^{k-1} } \right)  \abs{R/\pi}^{ (k-1) n/2}.  \] 
Averaging over all possible systems of representatives, we get \eqref{eq-first-phase-even}.

\end{proof}

By a second stationary phase analysis, we show cancellation occurs whenever $(a,x)\notin S$.

\begin{lem}\label{second-phase-even-lem} For $n$ even and $(a_0,x)\in (R/\pi^n)^2$, we have  \[ \sum_{ \substack{ a \in R/\pi^n R \\ a \equiv a_0 \bmod \pi^c \\  a^k \equiv x \bmod \pi^{n/2}}} \psi\left ( (k-1) a + \frac{x}{a^{k-1} } \right)  =0 \] if $(a_0,x)\notin \mathcal S$, and this sum equals $\abs{R/\pi}^{ n-c}   \psi\left ( (k-1) a_0 + \frac{x}{a_0^{k-1} } \right)  $ if $(a_0,x)\in \mathcal S$.\end{lem}

\begin{proof} By \cref{kth-power-depends}, the condition $a^k \equiv x \bmod \pi^{n/2}$ depends only on $a$ mod $\pi^c$.

Thus if $a_0^k \equiv x\bmod \pi^{n/2}$, the sum simplifies as
\[ \sum_{ \substack{ a \in R/\pi^{n}R  \\ a \equiv a_0 \bmod \pi^c }} \psi\left(  (k-1) a+ \frac{x}{a ^{k-1} }\right) = \sum_{ y \in \pi^c R/ \pi^{n/2}R } \psi\left(  (k-1) (a_0+y) + \frac{x}{(a_0+y) ^{k-1} }\right)\] and otherwise the sum vanishes. If $a_0^k \equiv x \bmod \pi^{n/2}$ then $(a_0, x)\notin \mathcal S$ by \cref{E-to-power} and the claim is automatically true, so we may assume $a_0^k \equiv x\bmod \pi^{n/2}$.

Now by \cref{homomorphism-equation}, $ (k-1) (a_0+y) + \frac{x}{(a_0+y) ^{k-1} }$ is a group homomorphism $\pi^c R \to R / \pi^n$ plus a constant. Thus $ \psi\left(  (k-1) (a_0+y) + \frac{x}{(a_0+y) ^{k-1} }\right)$ is an additive character of $y$ times a constant. Hence the sum vanishes unless this additive character is trivial. This occurs exactly when $(a_0,x)\in \mathcal S$.  \end{proof}

\begin{lem}\label{second-phase-even} For $n$ even, we have \[ Kl_k(x) = \sum_{ \substack{a \in R/ \pi^n \\ (a,x) \in \mathcal S } }\psi\left ( (k-1) a + \frac{x}{a^{k-1} } \right)   \abs{R/\pi}^{(k-2)n/2} .\] \end{lem}

\begin{proof} This follows from \cref{first-phase-even} and \cref{second-phase-even-lem}. \end{proof}

We can immediately deduce a slightly weaker form of our main bound in the even case:

\begin{lem}\label{final-even} For $n$ even, we have \[ \abs{Kl_k(x)} \leq k'  \abs{R/\pi}^{k n/2 - \tilde{c} }. \] \end{lem}

\begin{proof} This follows from combining \cref{second-phase-even} and \cref{a-counting}. \end{proof}

When $c > \tilde{c}$, we must improve this slightly. 
 
\begin{lem}\label{second-phase-even-gauss-lem} Fix $(a_0, x) \in (R/\pi^n R)^2$. For  $c \neq \tilde{c}$, we have  \[  \Biggl| \sum_{ \substack{ a \in R/\pi^n R \\ a \equiv a_0 \bmod \pi^{\tilde{c}} \\  (a,x) \in \mathcal S } }\psi\left ( (k-1) a + \frac{x}{a^{k-1} } \right) \Biggr| \leq \begin{cases} \sqrt{k^*/k'}  \abs{R/\pi}^{ n- \frac{c}{2} - \frac{ \tilde{c}}{2} }   & \textrm{if } (a_0,x)\in \mathcal S \\ 0 & \textrm{otherwise}\end{cases}.\]
\end{lem}

\begin{proof} Since $c\neq \tilde{c}$, we must have $c= \tilde{c}+1$.

 By \cref{a-uniqueness}, whether $(a,x)\in \mathcal S$ depends only on $a$ modulo $\pi^{ \tilde{c}}$, so the sum is empty and the result is trivial if $(a_0,x)\notin\mathcal S$, and if $(a_0,x)\in \mathcal S$, then $(a,x)\in \mathcal S$ for every $a$ in the sum.  In particular, this implies  \[\psi\left ( (k-1) (a \pi^{\tilde{c}} t) + \frac{x}{(a+\pi^{ \tilde{c}}t)^{k-1} } \right) \] depends only on $t$ mod $\pi$.  Define $\varphi \colon R/\pi \to \{ z\in \mathbb C \mid \abs{z}=1\}$ by 
 \[\varphi(t) = \psi\left ( (k-1) (a_0 \pi^{\tilde{c}} t) +  \frac{x}{(a_0+ \pi^{\tilde{c}}t)^{k-1} } \right).\]
 
 Then 
\[ \sum_{ \substack{ a \in R/\pi^n R \\ a \equiv a_0 \bmod \pi^{\tilde{c}} \\  (a,x) \in \mathcal S}} \psi\left ( (k-1) a + \frac{x}{a^{k-1} } \right)  =\sum_{t \in R/\pi} \varphi(t) \abs{R/\pi}^{n-c}.  \]
In the notation of \cref{general-gauss-bound}, we have

\[ \widetilde{\varphi}(t_1,t_2) = \psi \left(  \frac{x}{(a_0+ \pi^{\tilde{c}}(t_1+t_2) )^{k-1}} -  \frac{x}{(a_0+ \pi^{\tilde{c}}t_2)^{k-1} }- \frac{x}{(a_0+ \pi^{\tilde{c}}t_2)^{k-1}} - \frac{x}{(a_0)^{k-1} }\right)\] \[=  \psi \Bigl( \sum_{i_1,i_2=1}^{\infty} (-1)^{i_1+i_2}  \binom{k + i_1 + i_2 -2}{i_1,i_2, k-2} \pi^{\tilde{c} (i_1+i_2) } t_1^{i_1} t_2^{i_2} a_0^{ 1-k-i_1-i_2}  \Bigr)\] by \cref{hom-to-multinom}. By \cref{multinom-bound} and \eqref{c-tilde-def} every term is divisible by $\pi^{n-1}$, and furthermore is divisible by $\pi^n$ unless $i_1,i_2 = (1, p^r)$ or $(p^r,1)$. Since $t^{p^r}$ is an additive polynomial in $t$, it follows that $\widetilde{\varphi}$ is a group homomorphism in each variable. So we may apply \cref{general-gauss-bound}. 

Here $W$ consists of exactly those $t_1$ so that \[ \psi \Bigl( \sum_{i_1,i_2=1}^{\infty} (-1)^{i_1+i_2}  \binom{k + i_1 + i_2 -2}{i_1,i_2, k-2} \pi^{\tilde{c} (i_1+i_2) } t_1^{i_1} t_2^{i_2} a_0^{ 1-k-i_1-i_2}  \Bigr) =1 \] for all $t_2 \in R/\pi$. Equivalently, these are $t_1$ such that \[ \psi \Bigl( \sum_{i_1,i_2=1}^{\infty} (-1)^{i_1+i_2}  \binom{k + i_1 + i_2 -2}{i_1,i_2, k-2} (\pi^{\tilde{c} } t_1)^{i_1} y^{i_2} a_0^{ 1-k-i_1-i_2}  \Bigr) =1 \] for all $y\in \pi^{\tilde{c}}R$.

By \cref{a-shift-valuation}(ii), this can only happen for $t_1\neq 0$ if $R$ is a ring of mixed characteristic and $\tilde{c}= \frac{e}{p^{j+1}- p^j}$ for some $j$ from $0$ to $v-1$. Furthermore, in that case there are at most $p$ possible values of $ t_1$. Thus $\abs{W}=1$ unless $\tilde{c}= \frac{e}{p^{j+1}- p^j}$ and $\abs{W} \leq p$ in that case.

So \cref{general-gauss-bound} implies that  \[ \Bigl| \sum_{t \in R/\pi} \varphi(t) \bigr|   \leq \begin{cases}  \sqrt{p} \abs{R/\pi} ^{ \frac{1}{2}} & \textrm{if } \tilde{c}= \frac{e}{p^{j+1}- p^j}\textrm{ for some } 0\leq j\leq v-1\\ \abs{R/\pi} ^{ \frac{1}{2}} & \textrm{otherwise} \end{cases} .\]

If $\tilde{c}\neq \frac{e}{p^{j+1}- p^j}$ for all $0 \leq j \leq v-1$, we obtain \[ \Biggl| \sum_{ \substack{ a \in R/\pi^n R \\ a \equiv a_0 \bmod \pi^{\tilde{c}} \\  (a,x) \in \mathcal S }} \psi\left ( (k-1) a + \frac{x}{a^{k-1} } \right)  \Biggr | \leq    \abs{R/\pi}^{ n- \frac{c}{2} - \frac{\tilde{c}}{2} } \] which gives the desired bound since $k'\leq k^*$.

On the other hand, if $\tilde{c}= \frac{e}{p^{j+1}- p^j} $, we have  \[ (n-1) \leq (p^j+1) \tilde{c} + e (v-j) =  \frac{e (p^j+1) }{p^{j+1}- p^j}+ e (v-j) = e \left( v-j +  \frac{p^j+1 }{p^{j+1}- p^j}\right) \] and because $\tilde{c}<c $,
\[ n-1 \geq  (p^r+1) \tilde{c} + e (v-r)   =  \frac{e (p^r+1) }{p^{j+1}- p^j}+ e (v-r)  \geq  \frac{e (p^j+1) }{p^{j+1}- p^j}+ e (v-j) = e \left( v-j +  \frac{p^j+1 }{p^{j+1}- p^j}\right) \]  (because increasing $r$ by one increases $\frac{e (p^r+1) }{p^{j+1}- p^j}+ e (v-r) $ by $\left( \frac{p^{r+1}-p^r}{p^{j+1}-p^j}-1\right) $ which is $\leq 0$ if $r\leq j$ and $\geq 0$ if $r\geq j)$. Thus $ e \left( v-j +  \frac{p^j+1 }{p^{j+1}- p^j}\right)= n-1$, which means that $w- w'=1$ by \eqref{w-def} and \eqref{w'-def} and thus $\frac{k'}{k^*}=p$, giving the desired bound also in this case.

\end{proof}

\begin{prop}\label{final-even-2} For $n$ even, we have \[ \abs{Kl_k(x)} \leq k^*   \abs{R/\pi}^{ \frac{kn- c-  \tilde{c} }{2} }. \] \end{prop}

\begin{proof} If $c =\tilde{c}$ then this follows from \cref{final-even} and $k'\leq k^*$. Otherwise, it follows by combining \cref{second-phase-even}, \cref{second-phase-even-gauss-lem}, and \cref{a-counting}. \end{proof}

We now begin the odd case in the same way as the even.

\begin{lem}\label{first-phase-odd}For $n$ odd we have \[Kl_k(x) = \sum_{\substack{ x_1,\dots,x_k  \in R/\pi^nR \\ \prod_{i=1}^k x_i =x \\ x_1 \equiv x_2 \equiv \dots \equiv x_k \bmod \pi^{ \frac{n-1}{2}}}} \psi \Bigl( \sum_{i=1}^k x_i\Bigr) .\] \end{lem}

\begin{proof} Pick a set $S$ of representatives of congruence classes in $R/\pi^{\frac{n+1}{2}}R$. Write each $x_i$ as $a_i+b_i$ where $a_i\in S$ and $b_i$ is divisible by $\pi^{\frac{n+1}{2}}$.

Then\[Kl_k(x) =\sum_{ \substack {a_1,\dots, a_k \in S\\ \prod_{i=1}^k a_i \equiv x \bmod \pi^{ \frac{n+1}{2}} }}   \sum_{\substack{ b_1,\dots,b_k   \in \pi^{ \frac{n+1}{2}} R/\pi^{ n} R\\ \prod_{i=1}^k (a_i+b_i)  =x }} \psi \Bigl( \sum_{i=1}^k a_i+ \sum_{i=1}^k b_i\Bigr ) .\]

Since $b_i b_j = 0 $ for all $i,j$, the equation $ \prod_{i=1}^k (a_i+b_i)  =x $ simplifies to \begin{equation}\label{aff-eq-odd} x=   \Bigl( 1+ \sum_{i=1}^k\frac{b_i}{a_i} \Bigr) \prod_{i=1}^k a_i  .\end{equation}  The sum over $b_i$ vanishes unless the character $\psi \Bigl( \sum_{i=1}^k a_i+ \sum_{i=1}^k b_i\Bigr ) $ is constant over the affine hyperplane of solutions $(b_1,\dots, b_k)$ to \eqref{aff-eq-odd}, which occurs only if $a_1 \equiv a_2 \equiv \dots \equiv a_k \bmod \pi^{\frac{n-1}{2}}$ because otherwise we can add a multiple of $a_i$ to $b_i$ and subtract the same multiple of $a_j$ from $b_j$ to change the value of the character. Thus
\[Kl_k(x) =\sum_{ \substack {a_1,\dots, a_k \in S\\ \prod_{i=1}^k a_i \equiv x \bmod \pi^{ \frac{n+1}{2}}  \\ a_1 \equiv a_2 \equiv \dots \equiv a_k \bmod \pi^{\frac{n-1}{2}} }}   \sum_{\substack{ b_1,\dots,b_k   \in \pi^{ \frac{n+1}{2}} R/\pi^{ n} R\\ \prod_{i=1}^k (a_i+b_i)  =x }} \psi \Bigl( \sum_{i=1}^k a_i+ \sum_{i=1}^k b_i\Bigr ) \]
\[ =  \sum_{\substack{ x_1,\dots,x_k  \in R/\pi^nR \\ \prod_{i=1}^k x_i =x \\ x_1 \equiv x_2 \equiv \dots \equiv x_k \bmod \pi^{ \frac{n-1}{2}} }} \psi \Bigl( \sum_{i=1}^k x_i\Bigr).\] \end{proof}

Define the Gauss sum
\[ G_k (  \alpha, \beta) = \sum_{ \delta_1,\dots, \delta_{k-1} \in R/\pi R} \psi\Bigl( \pi^{n-1} \Bigl( \alpha \sum_{i=1}^{k-1} \delta_i + \beta \sum_{1 \leq i \leq j \leq k-1} \delta_i \delta_j \Bigr) \Bigr) \]
where $\alpha, \beta \in R/\pi R$.

\begin{lem}\label{to-Gauss-odd} For $n>1$ odd, we have \[ Kl_k( x) = \sum_{\substack{a \in R/\pi^n R \\ a^k \equiv x\bmod  \pi^{ \frac{n-1}{2}}}}  \psi \left(  (k-1)a + \frac{x}{a^{k-1}} \right) G_k \left( \frac{a^k-x}{ x \pi^{\frac{n-1}{2}}}, \frac{1}{a} \right) \abs{R/\pi}^{ \frac{ (n-1)(k-1) - n -1 }{2}} .\] \end{lem}

\begin{proof} For each $x_1,\dots, x_k$ such that $x_1 \equiv x_2 \equiv \dots \equiv x_k \bmod \pi^{ \frac{n-1}{2}}$ there exist exactly $\abs{R/\pi}^{\frac{n+1}{2}}$ values of $a \in R/\pi^nR$ such that $a \equiv x_1 \equiv x_2 \equiv \dots \equiv x_k \bmod \pi^{ \frac{n-1}{2}}$. This, combined with \cref{first-phase-odd}, gives   
\[Kl_k(x) = \sum_{\substack{ a, x_1,\dots,x_k  \in R/\pi^nR \\ \prod_{i=1}^k x_i =x \\ a \equiv x_i  \bmod \pi^{ \frac{n-1}{2}} \textrm{ for all } i }} \psi \Bigl( \sum_{i=1}^k x_i\Bigr)\frac{1}{ \abs{R/\pi}^{\frac{n+1}{2}}} .\] 
For this condition to be satisfied, we must have $a^k \equiv x \bmod \pi^{ \frac{n-1}{2}}$. When this is satisfied, we can write each $x_i$ uniquely as $a_i +  \pi^{ \frac{n-1}{2}} b_i$ for some $b_i \in R / \pi^{ \frac{n+1}{2}} R$. This gives
\[Kl_k(x) =\sum_{\substack{a \in R/\pi^n R \\ a^k \equiv x\bmod  \pi^{ \frac{n-1}{2}}}}  \sum_{\substack{ b_1,\dots, b_k \in R/\pi^{\frac{n+1}{2}}R \\ \prod_{i=1}^k (a +  \pi^{ \frac{n-1}{2}} b_i) =x }} \psi \Bigl( \sum_{i=1}^k (a +  \pi^{ \frac{n-1}{2}} b_i) \Bigr)\frac{1}{ \abs{R/\pi}^{\frac{n+1}{2}}} .\]  
Now
 \[ \sum_{i=1}^k (a +  \pi^{ \frac{n-1}{2}} b_i) = \sum_{i=1}^{k-1} (a+  \pi^{ \frac{n-1}{2}} b_i) + \frac{x}{ \prod_{i=1}^{k-1}( a +  \pi^{ \frac{n-1}{2}} b_i )}\] \[ =  \sum_{i=1}^{k-1} (a +  \pi^{ \frac{n-1}{2}} b_i) + \frac{x}{ a^{k-1}}  - \sum_{i=1}^{k-1} \frac{ x  \pi^{ \frac{n-1}{2}} b_i }{ a^k} + \sum_{1 \leq i \leq j \leq k-1} \frac{ x \pi^{n-1} b_i b_j }{ a^{k+1}} \]
 \[ = (k-1)a + \frac{x}{a^{k-1}} + \pi^{\frac{n-1}{2}} \left(1 - \frac{x}{a^k} \right) \sum_{i=1}^{k-1} b_i + \pi^{n-1} \frac{x}{ a^{k+1}}  \sum_{1 \leq i \leq j \leq k-1}  b_i b_j \] where we may truncate the Taylor expansion to second-order since the higher-order terms are divisible by $\pi^{ \frac{3 (n-1)}{2} }$ and $ \frac{3 (n-1)}{2}  \geq n$ because $n\geq 3$. Furthermore $b_k$ is uniquely determined by $b_1,\dots, b_{k-1} $ and the equation $\prod_{i=1}^k (a +  \pi^{ \frac{n-1}{2}} b_i) =x$. This gives
\[Kl_k(x) = \sum_{\substack{a \in R/\pi^n R \\ a^k \equiv x\bmod  \pi^{ \frac{n-1}{2}}}}  \psi \left(  (k-1)a + \frac{x}{a^{k-1}} \right) \times\]
\[ \sum_{ b_1,\dots, b_{k-1}  \in R/\pi^{\frac{n+1}{2}}R } \psi \left(  \pi^{\frac{n-1}{2}} \left(1 - \frac{x}{a^k} \right) \sum_{i=1}^{k-1} b_i + \pi^{n-1} \frac{x}{ a^{k+1}}  \sum_{1 \leq i \leq j \leq k-1}  b_i b_j \right) \frac{1}{ \abs{R/\pi}^{\frac{n+1}{2}}}.\]

Next note that $a^k$ is congruent to $x$ modulo $\pi^{\frac{n-1}{2}}$ and so  $1 - \frac{x}{a^k}$ is divisible by $\pi^{\frac{n-1}{2}}$ and thus $ \pi^{\frac{n-1}{2}} \left(1 - \frac{x}{a^k} \right)$ is divisible by $\pi^{n-1}$. Since each coefficient is divisible by $\pi^{n-1} $,  the term summed over $b_i$ depends only on $b_i$ modulo $\pi$. Since for each $i$, each residue class mod $\pi$ occurs for $\abs{R/\pi}^{ \frac{n-1}{2}}$ possible $b_i$, 
\[  \sum_{ b_1,\dots, b_{k-1}  \in R/\pi^{\frac{n+1}{2}}R } \psi \left(  \pi^{\frac{n-1}{2}} \left(1 - \frac{x}{a^k} \right) \sum_{i=1}^{k-1} b_i + \pi^{n-1} \frac{x}{ a^{k+1}}  \sum_{1 \leq i \leq j \leq k-1}  b_i b_j \right)  \]
\[ =  \abs{R/\pi}^{ \frac{ (n-1)(k-1)}{2}} \sum_{ \delta_1,\dots,\delta_{k-1}  \in R/\pi R} \psi \left(  \pi^{\frac{n-1}{2}} \left(1 - \frac{x}{a^k} \right) \sum_{i=1}^{k-1} \delta_i + \pi^{n-1} \frac{x}{ a^{k+1}}  \sum_{1 \leq i \leq j \leq k-1}  \delta_i \delta_j \right) \] \[=  \abs{R/\pi}^{ \frac{ (n-1)(k-1)}{2}} G_k \left( \frac{a^k-x}{ a^k \pi^{\frac{n-1}{2}}}, \frac{x}{a^{k+1}} \right) =  \abs{R/\pi}^{ \frac{ (n-1)(k-1)}{2}} G_k \left( \frac{a^k-x}{ x \pi^{\frac{n-1}{2}}}, \frac{1}{a} \right) \]
which gives
\[ Kl_k( x) = \sum_{\substack{a \in R/\pi^n R \\ a^k \equiv x\bmod  \pi^{ \frac{n-1}{2}}}}  \psi \left(  (k-1)a + \frac{x}{a^{k-1}} \right) G_k \left( \frac{a^k-x}{ x \pi^{\frac{n-1}{2}}}, \frac{1}{a} \right) \abs{R/\pi}^{ \frac{ (n-1)(k-1) - n -1 }{2}} . \qedhere\] \end{proof} 

\begin{lem}\label{second-phase-odd-lem} For $n>1$ odd and $(a_0,x) \in (R/\pi^n R)^2$, we have  \[ \sum_{ \substack{ a \in R/\pi^n R \\ a \equiv a_0 \bmod \pi^c \\  a^k \equiv x \bmod \pi^{\frac{n-1}{2}}}} \psi\left ( (k-1) a + \frac{x}{a^{k-1} } \right)G_k \left( \frac{a^k-x}{ x \pi^{\frac{n-1}{2}}}, \frac{1}{a} \right)  =0 \] if $(a_0,x) \notin \mathcal S$, and this sum equals $\abs{R/\pi}^{ n-c}   \psi\left ( (k-1) a_0 + \frac{x}{a_0^{k-1} } \right) G_k \left( \frac{a_0^k-x}{ x \pi^{\frac{n-1}{2}}}, \frac{1}{a_0} \right)  $ if $(a_0,x)\in \mathcal S$.\end{lem}

\begin{proof} By \cref{kth-power-depends}, the condition $a^k \equiv x \bmod \pi^{\frac{n-1}{2}}$ depends only on $a$ mod $\pi^c$. Furthermore, by the same lemma, the congruence class of $\frac{a^k-x}{ x \pi^{\frac{n-1}{2}}}$ mod $\pi$ depends only on $a$ mod $\pi^c$, and, since $c\geq 1$, $\frac{1}{a}$ mod $\pi$ depends only on $a \bmod \pi^c$, so  $G_k \left( \frac{a^k-x}{ x \pi^{\frac{n-1}{2}}}, \frac{1}{a} \right)$ depends only on $a$ mod $\pi^c$.

Thus if $a_0^k \equiv x\bmod \pi^{\frac{n-1}{2}}$, the sum simplifies as
\[G_k \left( \frac{a_0^k-x}{ x \pi^{\frac{n-1}{2}}}, \frac{1}{a_0} \right)   \sum_{ \substack{ a \in R/\pi^{n}R  \\ a \equiv a_0 \bmod \pi^c }} \psi\left(  (k-1) a+ \frac{x}{a ^{k-1} }\right)  \] \[=G_k \left( \frac{a_0^k-x}{ x \pi^{\frac{n-1}{2}}}, \frac{1}{a_0} \right)  \sum_{ y \in \pi^c R/ \pi^{n/2}R } \psi\left(  (k-1) (a_0+y) + \frac{x}{(a_0+y) ^{k-1} }\right)\] and otherwise the sum vanishes. If $a_0^k \not\equiv x \bmod \pi^{\frac{n-1}{2}}$ then $(a_0,x)\notin \mathcal S$ is not satisfied by \cref{E-to-power} and the claim is automatically true, so we may assume $a_0^k \equiv x\bmod \pi^{\frac{n-1}{2}}$.

Now by \cref{homomorphism-equation}, $ (k-1) (a_0+y) + \frac{x}{(a_0+y) ^{k-1} }$ is a group homomorphism $\pi^c R \to R / \pi^n$ plus a constant. Thus $ \psi\left(  (k-1) (a_0+y) + \frac{x}{(a_0+y) ^{k-1} }\right)$ is an additive character of $y$ times a constant. Hence the sum vanishes unless this additive character is trivial. This occurs exactly when $(a_0,x)\in \mathcal S$.  \end{proof}

\begin{lem}\label{second-phase-odd} For $n>1$ odd, we have \[ Kl_k(x) = \sum_{ \substack{a \in R/ \pi^n \\ (a,x)\in \mathcal S }} \psi\left ( (k-1) a + \frac{x}{a^{k-1} } \right) G_k \left( \frac{a^k-x}{ x \pi^{\frac{n-1}{2}}}, \frac{1}{a} \right) \abs{R/\pi}^{ \frac{ (n-1)(k-1) - n -1 }{2}}.\] \end{lem}

\begin{proof} This follows from \cref{to-Gauss-odd} and \cref{second-phase-odd-lem}. \end{proof}

Next, we will need to understand the Gauss sum $G_k(\alpha, \beta)$.

\begin{lem}\label{odd-gauss-bound} Fix $\alpha, \beta \in R/\pi R$ with $\beta \neq 0$. If $p\nmid k$ then \[ \abs{G_k(\alpha,\beta)} = \abs{R/\pi}^{\frac{k-1}{2} } \] and if $p \mid k$ and $p$ is odd or $k$ is a multiple of $4$ then  \[ \abs{G_k(\alpha,\beta)} = \begin{cases} \abs{R/\pi}^{\frac{k}{2} }&\textrm{if } \alpha = 0 \\ 0 &\textrm{if } \alpha\neq0  \end{cases}  \]
while if $p=2$, $2\mid k$, and $4 \nmid k$, we have 
\[ \abs{G_k(\alpha,\beta)} = \begin{cases} \abs{R/\pi}^{\frac{k}{2} }&\textrm{if } \alpha^2 = \lambda^2 \beta\\ 0 &\textrm{if } \alpha^2\neq \lambda^2 \beta  \end{cases}  \] where $\lambda \in R/\pi R$ is the unique element satisfying $\psi ( \pi^{n-1}  x^2) = \psi( \pi^{n-1} \lambda x)$ for all $x$.
 \end{lem}

\begin{proof} We use \cref{general-gauss-bound}, applied to the phase \[ \varphi (\boldsymbol \delta )  = \psi (\pi^{n-1} Q( \boldsymbol \delta)) \] where \[ Q(\boldsymbol \delta) =  \sum_{i=1}^{k-1} \delta_i + \beta \sum_{1 \leq i \leq j \leq k-1} \delta_i \delta_j.\] whose associated bilinear form is 
\[ B( \boldsymbol \gamma, \boldsymbol \delta )  = Q( \boldsymbol \gamma + \boldsymbol \delta) - Q(\boldsymbol \gamma) - Q(\boldsymbol \delta) + Q(0) \] \[= \beta \sum_{1 \leq i \leq j \leq k-1} \left((\delta_i +\gamma_i)(\delta_j+\gamma_j) - \delta_i \delta_j -\gamma_i \gamma_j + 0 \right) =  \beta \sum_{1 \leq i \leq j \leq k-1} \left(\delta_i \gamma_j + \delta_j \gamma_i \right)\] \[=\beta  \sum_{1\leq i,j\leq k-1} \delta_i \gamma_j + \beta \sum_{1\leq i \leq k-1} \delta_i \gamma_i .\]
Viewing symmetric bilinear forms as arising from symmetric matrices in the usual way, the second term arises from $\beta$ times the identity matrix while the first arises from $\beta$ times the all $1$s matrix. The all-ones matrix has one eigenvalue $k-1$ and the rest $0$, and adding the identity matrix gives one eigenvalue $k$ and the rest $1$, while multiplying by $\beta$ gives one eigenvalue $\beta k$ and the rest $\beta$.

Since $\beta \neq 0 $, we see if $p \nmid k$ that $B$ is nondegenerate and so $W =0$. This gives the estimate in the first case.

If $p \mid k$, this matrix has eigenvalue $0$ with multiplicity one and thus its kernel is one-dimensional. We can see immediately that the kernel is generated by the all $1$s vector, i.e. consists of vectors with $\delta_i = \delta$ for all $i$. Thus, $W$ is the subspace generated by the all-$1$s vector, and we obtain an estimate $q^{ \frac{k}{2}}$ if $  \psi (\pi^{n-1} Q( \boldsymbol \delta)) $ is constant on $W$ and $0$ otherwise. It remains to determine when this restriction is constant.

 Restricting $Q$ to $W$, we get
\[ Q(\delta, \dots, \delta) =\alpha (k-1) \delta+ \beta \binom{k}{2} \delta^2 .\]
If $p$ is odd or $p=2$ and $k$ is a multiple of $4$ then $p$ divides $\binom{k}{2}$ so  $Q(\delta,\dots, \delta) = -\alpha \delta$ and thus $ \psi (\pi^{n-1} Q( \boldsymbol \delta))$ is constant if and only if $\alpha=0$.

If $p=2$ and $k$ is not a multiple of $4$ then $\binom{k}{2} \equiv 1\bmod 2$ so $Q(\delta,\dots, \delta)=\alpha \delta + \beta \delta^2$, and, after composing with $\psi(\pi^{n-1} (\cdot))$, we get \[ \psi (\pi^{n-1} ( \alpha \delta + \beta \delta^2)) = \psi(\pi^{n-1} ( \alpha + \lambda\sqrt{\beta})\delta)\] which is constant if and only if $\alpha + \lambda\sqrt{\beta}=0$, which happens if and only if $\alpha^2=\lambda^2\beta$.
 \end{proof}
 
 \begin{lem}\label{final-odd} For $n>1$ odd, we have \[ \abs{Kl_k(x)} \leq k'   \abs{R/\pi}^{k n/2 - \tilde{c} }. \] \end{lem}

\begin{proof} By \cref{second-phase-odd},  \cref{odd-gauss-bound}, and \cref{a-counting}, we have
\[\abs{ Kl_k(x)} = \Biggl| \sum_{ \substack{a \in R/ \pi^n \\ (a,x) \in \mathcal S}} \psi\left ( (k-1) a + \frac{x}{a^{k-1} } \right)  G_k \left( \frac{a^k-x}{ x \pi^{\frac{n-1}{2}}}, \frac{1}{a} \right) \abs{R/\pi}^{ \frac{ (n-1)(k-1) - n -1 }{2}}\Biggr|\] 
\[ \leq \sum_{ \substack{a \in R/ \pi^n \\ (a,x) \in \mathcal S}}  \abs{R/\pi}^{ \frac{k}{2} }\abs{R/\pi}^{ \frac{ (n-1)(k-1) - n -1 }{2}}  \leq k' \abs{R/\pi}^{n- \tilde{c}}  \abs{R/\pi}^{ \frac{k}{2} }\abs{R/\pi}^{ \frac{ (n-1)(k-1) - n -1 } {2}}\] \[= k' \abs{R/\pi}^{ \frac{ nk}{2} - \tilde{c}}. \]
 \end{proof}

 Again, a slight improvement can be made if $c>\tilde{c}$.
 
 \begin{prop}\label{final-odd-2} For $n>1$ odd, we have \[ \abs{Kl_k(x)} \leq k^*   \abs{R/\pi}^{ \frac{kn- c-  \tilde{c} }{2} }. \] \end{prop}
 
 \begin{proof} If $c = \tilde{c}$ this follows from \cref{final-odd} and the bound $k'\leq k^*$. If $c \neq \tilde{c}$ then $c= \tilde{c}+1$.
 
 If $k$ is not divisible by $p$ then we repeat the argument of \cref{final-odd}, saving an additional factor of $\abs{R/\pi}^{\frac{1}{2}}$ in the application of \cref{odd-gauss-bound}, obtaining the conclusion since $c = \tilde{c}+1$. 
 
 If $c=\tilde{c}+1$ and $k$ is divisible by $p$, by the second case of \cref{kth-power-depends}, $\frac{a^k-x}{ x \pi^{\frac{n-1}{2}}} \bmod \pi$ depends only on $a$ mod $\pi^{\tilde{c}}$. The same is true for $\frac{1}{a} \bmod \pi$, so
 $G_k \left( \frac{a^k-x}{ x \pi^{\frac{n-1}{2}}}, \frac{1}{a} \right)$ depends only on $a$ modulo $\pi^{ \tilde{c}}$.

Hence we can apply \cref{second-phase-even-gauss-lem} to obtain   \[\abs{ Kl_k(x)} = \Biggl| \sum_{ \substack{a \in R/ \pi^n \\ (a,x) \in \mathcal S }} \psi\left ( (k-1) a + \frac{x}{a^{k-1} } \right)  G_k \left( \frac{a^k-x}{ x \pi^{\frac{n-1}{2}}}, \frac{1}{a} \right) \abs{R/\pi}^{ \frac{ (n-1)(k-1) - n -1 }{2}}\Biggr|\] 
\[ \leq   \sum_{ \substack{a \in R/ \pi^{\tilde{c}}  \\ (a,x) \in \mathcal S}}    \sqrt{k^*/k'}  \abs{R/\pi}^{ n- \frac{c}{2} - \frac{ \tilde{c}}{2} }  \abs{ G_k \left( \frac{a^k-x}{ x \pi^{\frac{n-1}{2}}}, \frac{1}{a} \right) }  \abs{R/\pi}^{ \frac{ (n-1)(k-1) - n -1 }{2}} \]
\[\leq  \sum_{ \substack{a \in R/ \pi^{\tilde{c}}  \\ (a,x)\in \mathcal S}}    \sqrt{k^*/k'}  \abs{R/\pi}^{ n- \frac{c}{2} - \frac{ \tilde{c}}{2} }  \abs{R/\pi }^{\frac{k}{2}}   \abs{R/\pi}^{ \frac{ (n-1)(k-1) - n -1 }{2}} \]
\[ \leq k'    \sqrt{k^*/k'}  \abs{R/\pi}^{ n- \frac{c}{2} - \frac{ \tilde{c}}{2} }  \abs{R/\pi }^{\frac{k}{2}}   \abs{R/\pi}^{ \frac{ (n-1)(k-1) - n -1 }{2}} \]
\[= \sqrt{ k^* k'}  \abs{R/\pi} ^{ \frac{ nk  -c - \tilde{c} }{2}}  ,\] giving the desired bound since $k^*\geq k'$.

\end{proof}

Finally, we prove the lower bound. To do this, we prove $Kl_k(x)$ vanishes for most $x$, and then evaluate the $\ell^2$ norm of $Kl_k$, showing it must take a large value on some point

\begin{lem}\label{number-of-nonzero-values} For $n\geq 2$, we have $Kl_k(x) =0$ for all but at most  $ \abs{R/\pi}^{ c + \tilde{c}-1 } (\abs{R/\pi}-1)$ values of $x$. \end{lem}

\begin{proof}  The size of $\mathcal S$ is at most $\abs{R/\pi}^{n-1} (\abs{R/\pi}-1)$ times the maximum over $a$ of the number of $x$ with $(a,x)\in \mathcal S$. By \cref{x-uniqueness}, this maximum is $\abs{R/\pi}^c$, so $\abs{\mathcal S}$ is at most $\abs{R/\pi}^{n+c-1} (\abs{R/\pi}-1)$. By \cref{a-uniqueness}, if $(a,x)\in \mathcal S$ for at least one $a$ then $(a,x)\in \mathcal S$, for at least $\pi^{n-\tilde{c}}$ values of $a$, so the number of $x$ with $(a,x)\in \mathcal S$ for at least one $a$ is at most $\abs{\mathcal S}$ divided by $\pi^{n-\tilde{c}}$, and thus at most $ \abs{R/\pi}^{ c + \tilde{c}-1 } (\abs{R/\pi}-1)$.

Finally, by \cref{second-phase-even} in the $n$ even case and \cref{second-phase-odd} in the $k$ odd case, $Kl_k(x)=0$ unless there is at least one $a$ with $(a,x)\in \mathcal S$. \end{proof}

\begin{prop}\label{lower-bound} For $n\geq 2$, we have $\abs{Kl_k(x)} > \abs{R/\pi}^{ \frac{kn-c-\tilde{c}}{2}} $ for at least one value of $x$. \end{prop}
\begin{proof} Otherwise, we would have 
\[ \sum_{ x\in (R/\pi^n) } \abs{Kl_k(x)}^2 \leq  \sum_{ \substack{ x\in (R/\pi^n) \\ Kl_k(x) \neq 0 } } \abs{R/\pi}^{ {kn-c-\tilde{c}}} < \abs{R/\pi}^{ kn}\]  by \cref{number-of-nonzero-values}.
On the other hand,
\[  \sum_{ x\in (R/\pi^n) } \abs{Kl_k(x)}^2  =  \abs{R/\pi}^{ kn}\] by opening the sum and eliminating variables in pairs. \end{proof}

\section{A uniform CFKRS heuristic for twisted moments} 
Let $\mathbb F_q$ be a finite field with $q$ elements and $\pi$ an irreducible polynomial in $\mathbb F_q[T]$. Recall that $\mathbb F_q[T]^+_{\pi'}$ is the set of monic polynomials relatively prime to $\pi$.

We give a prediction for the value of the twisted moment \eqref{general-L-moment} of $L$-functions of Dirichlet characters over $\mathbb F_q[T]$ to fixed modulus, in the depth aspect of large $n$, fixed $\pi$. Thus, we will always assume $n\geq 2$, but a similar prediction could also be given for small $n$.

To motivate this, note that orthogonality of characters gives, for $g,h\in \mathbb F_q[T]^+_{\pi'} $, that 
\[ \sum_{ \chi \in  \mathcal F_{\pi,n}} \chi(a) \chi(h) \overline{\chi(g)}  =0\] unless  $a \equiv \beta g/h \bmod \pi^{n-1}$ for some $\beta \in \mathbb F_q^\times.$ When $a \equiv \beta g/h \bmod \pi^{n-1}$ for some (necessarily unique) $\beta$, set
\begin{equation}\label{C-definition} C_{g,h} = \sum_{ \chi \in  \mathcal F_{\pi,n}} \chi(a) \chi(h) \overline{\chi(g)}   = \abs{\pi}^{n-2} \times \begin{cases} \frac{q-2}{q-1} & \textrm{if }\beta=1 \\ - \frac{1}{ q-1}  & \textrm{if } \beta\neq 1 \end{cases} \times  \begin{cases} (\abs{\pi} -1)^2 & \textrm{if } \alpha = \beta g /h \bmod \pi^{n} \\ - (\abs{\pi} - 1) & \textrm{if }\alpha \neq \beta g/ h \bmod \pi^{n} \end{cases}\end{equation}
%
by another orthogonality calculation. Also write  $N=n \deg \pi -1$. Let $\mathcal Q$ be the set of pairs $(g,h)\in (\mathbb F_q[T]^+_{\pi'})^2\times \mathbb F_q^\times$ with $\gcd(g,h)=1$ and $a \equiv \beta g/h \bmod \pi^{n-1}$.  Then we predict

\begin{pred}\label{conj-CFKRS} There exists $\delta>0$ such that for all $\alpha_1,\dots, \alpha_{2k}$ imaginary and $a\in (\mathbb F_q[T]/\pi^n)^\times$
\begin{equation}\label{CFKRS-like} \ \begin{split} & \sum_{ \chi \in  \mathcal F_{\pi,n}} \chi(a) \prod_{i=1}^k L(1/2 + \alpha_i, \chi) \overline{L(1/2+ \alpha_{k+i},\chi)}\\ & =  \sum_{ \substack{ (g,h) \in \mathcal Q \\ \abs{g} \abs{h}  \leq q^N /\abs{\pi}^2  } } \sum_{\substack{ S \subseteq \{1,\dots, 2k\} \\ \abs{S} =k } }q^{ \scalebox{0.7}{$N\displaystyle (  \sum_{i\in S} \alpha_i- \sum_{i=1}^{k} \alpha_i )$}} \sum_{\substack{ f_1,\dots, f_{2k}  \in \mathbb F_q[T]^+_{\pi'} \\ g\prod_{i\notin S} f_i = h \prod_{i\in S} f_i  }}  C_{g,h}  \prod_{i \in S} \abs{f_i}^{ -\frac{1}{2}-\alpha_i} \prod_{i\notin S} \abs{f_i}^{- \frac{1}{2}+ \alpha_i}  + O( \abs{\pi} ^{ (1-\delta) n } )\end{split} \end{equation} where the sum over $f_1,\dots, f_{2k}$ in the right-hand side is interpreted as a meromorphic function in $\alpha_1,\dots,\alpha_{2k}$, analytically continued from the domain where it is absolutely convergent.   \end{pred}

Moreover, we will be interested in the particular value of $\delta$ in \cref{conj-CFKRS}. If \eqref{CFKRS-like} holds for all $\delta<1/2$ then we say \eqref{CFKRS-like} admits square-root cancellation.

 \eqref{CFKRS-like} looks similar to the predictions of \cite{btb,cms} for similar moments, except that those works summed over the ``diagonal" $g\prod_{i\notin S} f_i = h \prod_{i\in S} f_i$ for a single pair $g,h$, while we sum over multiple diagonals.  In this section, we briefly explain this choice, then show that \eqref{CFKRS-like} admits square-root cancellation in the $k=1$ case. We omit the step-by-step derivation of \eqref{CFKRS-like} as it is relatively standard, except for the use of multiple diagonals.
 
When $a$ can be written as $g/h$ for $g,h$ small, one need only to consider the diagonal associated to $g,h$, but if the residue class $a$ has multiple representations as a ratio, there is no clear reason to prioritize one over another. Summing over multiple diagonals is the simplest way to incorporate them into the estimate. The fact that it works in $k=1$, as we will see below, is evidence that it is the right approach in general. Furthermore, one can see from the $k=1$ estimate that if we ignore one diagonal, then it will produce a larger-than-square-root error term, preventing us from obtaining uniform square-root cancellation, and explaining the error term found in \cite[Theorem 10]{cms}.

On the other hand, if we summed over all representations of $a$ as a ratio, our predicted main term would not necessarily be any simpler than the original moment problem.  So it is necessary to sum only over $g,h$ below some cutoff. We have chosen $\abs{g}\abs{h}  \leq q^N/ \abs{\pi}^2$ as our cutoff because it simplifies our calculation in the $k=1$ case. Any cutoff which is close to $N$ should do the trick. We also include the monicity and coprimality conditions to avoid double-counting.

A key advantage of this is that the number of diagonals we need to sum over to obtain the main term is only of logarithmic size. Indeed if $(g_1,h_1)$ and $(g_2,h_2)$ both satisfy the conditions in the sum of \eqref{CFKRS-like}, and in addition $\deg h_1=\deg h_2$, then $\beta_1 g_1/h_1 \equiv a \equiv \beta_2 g_2/h_2 \bmod \pi^{n-1}$ implies $ \pi^{n-1} \mid \beta_1 g_1 h_2 - \beta_2 g_2 h_1 $. Also \[\abs{g_1} \abs{h_2} = \abs{g_1}\abs{h_1} \leq q^N/\abs{\pi}^2<  \abs{\pi}^{n-1} \] and the same is true for $\abs{g_2} \abs{h_1}$, and these together give $\beta_1 g_1 h_2 = \beta_2 g_2 h_1 $, and then by coprimality and monicity we have $h_1=h_2, g_1=g_2,\beta_1=\beta_2$. So the number of possibilities is at most $(n-2 ) \deg \pi$.

Shifting the cutoff far below $q^N$ would cause us to miss diagonal contributions of above-square-root size, while shifting it far above $q^N$ would cause our ``main term" to be a sum of polynomially many diagonals each of below-square-root size. Both are undesirable.

 \subsection{The case $k=1$}
 
 We now establish \eqref{CFKRS-like} for all $\delta<1/2$ if $k=1$.  In fact, we will give an error term of $O ( n \abs{\pi} ^{\frac{  n}{2} } ) $ for fixed $\pi$. Our strategy is to express both sides (ignoring the error term on the right side) as polynomials in $q^{ - \alpha_1}$ and $q^{\alpha_2}$ and compare their coefficients. Since the variables $q^{-\alpha_1}$ and $q^{\alpha_2}$ have absolute value $1$, the difference between the polynomials is bounded by the sum over degrees $d_1,d_2$ of the difference between their coefficients. So it suffices to show the sum of the absolute values of the differences of the coefficients is $O ( n \abs{\pi}^{\frac{n}{2}})$.
 
 Let \[  a_d(\chi) = q^{-\frac{d}{2} }  \sum_{ \substack{f_1 \in \mathbb F_q[T]^+_{\pi'}  \\ \deg f=d}} \chi(f)\]  so that \[ L(s,\chi) = \sum_{d=0}^{ N} a_d q^{\frac{d}{2} -ds}\] and the functional equation, whose constant $\epsilon_\chi$ satisfies $\abs{\epsilon_\chi}=1$, implies $a_d = \epsilon_\chi  \overline{a_{N-d}}$. Let $A_d$ be the number of monic polynomials of degree $d$ prime to $\chi$. We have $A_d=0$ for $d<0$.

 We have
 \[ L(1/2+\alpha_1,\chi)  \overline{L(1/2+\alpha_2,\chi)} = \sum_{d_1=0}^{N}  \sum_{d_2=0}^N a_{d_1} (\chi) \overline{a_{d_2}( \chi) } q^{ - d_1 \alpha_1+ d_2 \alpha_2} \] so that
\begin{equation}\label{coefficient-formula}   \sum_{ \chi \in  \mathcal F_{\pi,n}} \chi(a) L(1/2+\alpha_1,\chi)  \overline{L(1/2+\alpha_2,\chi)} =\sum_{d_1=0}^{N}  \sum_{d_2=0}^N    \sum_{ \chi \in  \mathcal F_{\pi,n}}\chi(a)a_{d_1}(\chi)  \overline{a_{d_2})(\chi) } q^{ - d_1 \alpha_1+ d_2 \alpha_2} . \end{equation}

  \begin{lem}\label{coefficient-average-evaluation} For any $d_1,d_2\geq 0$, we have\[  \sum_{ \chi \in  \mathcal F_{\pi,n}} \chi(a)   a_{d_1}(\chi) \overline{a_{d_2}(\chi)} =  q^{ - \frac{d_1+d_2}{2}}   \sum_{ \substack{(g,h) \in \mathcal Q  \\  \deg g- \deg h = d_2-d_1    }}  C_{g,h}  A_{d_2 -\deg g} \]
  \end{lem}
  
  \begin{proof}We have
   \[  \sum_{ \chi \in  \mathcal F_{\pi,n}} \chi(a)   a_{d_1} \overline{a_{d_2}}  = \sum_{ \chi \in  \mathcal F_{\pi,n}} \chi(a)   q^{ - \frac{d_1+d_2}{2}}  \sum_{ \substack{f_1,f_2 \in \mathbb F_q[T]^+_{\pi'}  \\ \deg f_i=d_i }} \chi(f_1) \overline{\chi(f_2)}  .\]  
 Then \eqref{C-definition} gives
\[ \sum_{ \chi \in  \mathcal F_{\pi,n}} \chi(a)\chi(f_1) \overline{\chi(f_2)}   = \begin{cases} C_{f_2,f_1} &  \textrm{if }a \equiv \beta f_2/f_1 \bmod \pi^{n-1} \textrm{ for some } \beta\in \mathbb F_q^\times \\ 0 & \textrm{otherwise } \end{cases}\]
Letting $g=  f_2 / \gcd(f_1,f_2)$ and $h = f_1/ \gcd(f_1,f_2)$ then $g$ and $h$ are coprime to each other and $\pi$, monic, and satisfy $g/h= f_2/f_1$ so that $(g,h)\in \mathcal Q$. Furthermore, from any $(g,h)\in \mathcal Q$, we can make $f_2,f_1$ by multiplying by a polynomial of degree $e$ coprime to $\pi$, as long as $\deg g= d_2 - e$ and $\deg h = d_1-e$, so the number of terms $(f_1,f_2)$ that give any pair $(g,h)$ is $A_{d_2 -\deg g}$ as long as $d_2-d_1=\deg g -\deg h$. This gives the statement. \end{proof}

  On the other hand, we can evaluate the $k=1$ case of the inner sum on the right hand side of \eqref{CFKRS-like}.

 \begin{lem}  
\begin{equation}\label{sum-in-the-right} \sum_{\substack{ S \subseteq \{1, 2\} \\ \abs{S} =1 } }q^{ \scalebox{0.7}{$N\displaystyle ( \sum_{i\in S} \alpha_i - \sum_{i=1}^{1} \alpha_i  $ )}} \sum_{\substack{ f_1, f_{2}  \in \mathbb F_q[T]^+_{\pi'} \\ g\prod_{i\notin S} f_i = h \prod_{i\in S} f_i  }}   \prod_{i \in S} \abs{f_i}^{ -\frac{1}{2}-\alpha_i} \prod_{i\notin S} \abs{f_i}^{- \frac{1}{2}+ \alpha_i}  \end{equation}
is a polynomial in $q^{-\alpha_1}$ and $q^{\alpha_2}$ whose coefficient of $q^{-d_1 \alpha_1 + d_2 \alpha_2}$ is
\begin{equation}\label{coefficient-formula} \begin{cases} 0 & \textrm{if }  \deg g - \deg h \neq d_2-d_1 \\  q^{ - \frac{d_1+d_2}{2}} A_{d_2-\deg g}   & \textrm{if }  \deg g - \deg h = d_2-d_1\textrm{ and }  d_1 + d_2 \leq N  \\   q^{ \frac{d_1+d_2}{2}-N} A_{N-d_1-\deg g} & \textrm{if }  \deg g - \deg h = d_2-d_1\textrm{ and }  d_1 + d_2 > N \end{cases}\end{equation} \end{lem} 

\begin{proof} Since $S =\{1\}$ or $S=\{2\}$, \eqref{sum-in-the-right} equals
\[ \sum_{\substack{ f_1, f_{2}  \in \mathbb F_q[T]^+_{\pi'} \\ gf_2 = h f_1  }}  \abs{f_1}^{- \frac{1}{2} - \alpha_1} \abs{f_2}^{ - \frac{1}{2} +\alpha_2} +  q^{ N( \alpha_2-\alpha_1)} \sum_{\substack{ f_1, f_{2}  \in \mathbb F_q[T]^+_{\pi'} \\ gf_1 = h f_2  }}  \abs{f_1}^{- \frac{1}{2} + \alpha_1} \abs{f_2}^{ - \frac{1}{2} -\alpha_2}. \]
We may uniquely express  $f_1 = g m$ and $f_2 = hm$ in the first sum for some $m \in \mathbb F_q[T]^+_{\pi'} $, and $f_1=hm$, $f_2=gm$ similarly in the second sum. This gives
\[ =  \sum_{ \substack{m \in \mathbb F_q[t]^+_{\pi'}}} \abs{g}^{ - \frac{1}{2}  + \alpha_2} \abs{h}^{- \frac{1}{2} -\alpha_1 }   \abs{m}^{ -1 - \alpha_1 + \alpha_2}   +  q^{ N (\alpha_2 -\alpha_1)}   \sum_{ m \in \mathbb F_q[t]^+_{\pi'} } \abs{g}^{ - \frac{1}{2} + \alpha_1} \abs{h}^{- \frac{1}{2} - \alpha_2}   \abs{m}^{ -1 + \alpha_1 - \alpha_2} \]
\[ = \sum_{ e=0}^\infty  \abs{g}^{ - \frac{1}{2} + \alpha_2} \abs{h}^{- \frac{1}{2} - \alpha_1}  A_e  q^{ (-1 - \alpha_1 +\alpha_2)e} +  q^{ N (\alpha_2 -\alpha_1)}  \sum_{ e=0}^{\infty}  \abs{g}^{ - \frac{1}{2} + \alpha_1}  A_e \abs{h}^{- \frac{1}{2} -  \alpha_2}   q^{e(  -1 + \alpha_1-\alpha_2) } .\] 
A truncated version of this sum
\[ = \sum_{ \substack{e\leq  \frac{ N - \deg g -\deg h}{2}} }  \abs{g}^{ - \frac{1}{2}  + \alpha_2} \  \abs{h}^{- \frac{1}{2} - \alpha_1}  A_e  q^{ (-1 - \alpha_1 +\alpha_2)e} +  q^{ N (\alpha_2 -\alpha_1)}  \sum_{\substack{e<  \frac{ N - \deg g -\deg h}{2}} } \abs{g}^{ - \frac{1}{2} + \alpha_1}  A_e \abs{h}^{- \frac{1}{2} -  \alpha_2}   q^{e(  -1 + \alpha_1-\alpha_2) } .\] 
is easily seen to be a polynomial in $q^{-\alpha_1}$ and $q^{\alpha_2}$. Extracting the coefficients, we obtain \eqref{coefficient-formula}.

The remaining terms are given by
\[ = \sum_{ \substack{e>  \frac{ N - \deg g -\deg h}{2}} }  \abs{g}^{ - \frac{1}{2}  + \alpha_2} \ \abs{h}^{- \frac{1}{2} - \alpha_1}  A_e  q^{ (-1 - \alpha_1 +\alpha_2)e} +  q^{ N (\alpha_2 -\alpha_1)}  \sum_{\substack{e \geq  \frac{ N - \deg g -\deg h}{2}} } \abs{g}^{ - \frac{1}{2} + \alpha_1}  A_e \abs{h}^{- \frac{1}{2} -  \alpha_2}   q^{e(  -1 + \alpha_1-\alpha_2) } .\] 
Since $A_e = q^e (1-\abs{\pi}^{-1} )$ for $e  \geq \frac{ N - \deg g -\deg h}{2}  $, both sums are geometric series. Evaluating the geometric series as meromorphic functions, we see that they cancel each other. 
\end{proof}

Hence the right hand side of \eqref{CFKRS-like} (ignoring the big $O$ term) is a polynomial in $q^{-\alpha_1}$ and $q^{\alpha_2}$ whose coefficient of $q^{-d_1 \alpha_1 + d_2 \alpha_2}$ is
\begin{equation}\label{RHS-coefficient} \sum_{ \substack{ (g,h) \in \mathcal Q \\  \deg g - \deg h = d_2 -d_1 \\ \abs{g} \abs{h}  \leq q^N /\abs{\pi}^2  } } C_{g,h}  \begin{cases} q^{ - \frac{d_1+d_2}{2}} A_{d_2-\deg g} & \textrm{if } d_1+d_2 \leq N \\ q^{ \frac{d_1+d_2}{2}-N} A_{N-d_1-\deg g}  & \textrm{if } d_1+ d_2> N \end{cases}. \end{equation}

We now bound the differences between the coefficients.

For $d_1+d_2\leq N$, by \eqref{coefficient-formula} and Lemma \ref{coefficient-average-evaluation}, the coefficient of $q^{-d_1 \alpha_1 + d_2 \alpha_2}$ in the left-hand side of \eqref{CFKRS-like} is 
\[\sum_{ \substack{ (g,h) \in \mathcal Q \\  \deg g - \deg h = d_2 -d_1  } }  q^{ - \frac{d_1+d_2}{2}} A_{d_2-\deg g}\]
so by \eqref{RHS-coefficient} the difference of the coefficients is
\begin{equation}\label{first-coefficient-difference} \sum_{ \substack{ (g,h) \in \mathcal Q \\  \deg g - \deg h = d_2 -d_1\\  \abs{g} \abs{h}  > q^N /\abs{\pi}^2  }  }   C_{g,h}  q^{ - \frac{d_1+d_2}{2}} A_{d_2-\deg g}. \end{equation}
We have $ \abs{C_{g,h}} \leq \abs{\pi}^n$ and $\abs{A_e}  \leq q^e$ so that \[    q^{ - \frac{d_1+d_2}{2}} \abs{A_{d_2-\deg g}} \leq q^{ d_1 - \deg g- \frac{d_1+d_2}{2}}= q^{ \frac{d_1-d_2}{2}-\deg g}=q^{ \frac{\deg g -\deg h}{2}-\deg g}= q^{ -\frac{\deg g +\deg h}{2}} \leq \frac{\abs{\pi}}{q^{\frac{N}{2}}} = \frac{ \abs{\pi} q^{\frac{1}{2}}}{ \abs{\pi}^{\frac{n}[2}}.\]
Each pair $(g,h)\in \mathcal Q$ contributes to \eqref{first-coefficient-difference} for at most $\deg \pi$ pairs $d_1,d_2$, and only if $\deg g+\deg h \leq d_1+d_2 \leq N$, so the sum over $d_1+d_2\leq N$ of (the absolute value of) \eqref{first-coefficient-difference} is bounded by $ \deg \pi q^{\frac{1}{2}}  \abs{\pi}^{\frac{n}{2}+1}$ times the number of $(g,h)\in \mathcal Q$ for which $q^N /\abs{\pi}^2< \abs{g}\abs{h} \leq q^N$.

\begin{lem} The number of $(g,h)\in \mathcal Q$ for which $q^N /\abs{\pi}^2< \abs{g}\abs{h} \leq q^N$ is at most $n \deg \pi (q-1) \abs{\pi} $. \end{lem}

\begin{proof} For each pair $g,h$,  the congruence class of the ratio $g/h$ mod $\pi^{n}$ must reduce modulo $\pi^{n-1}$ to $a/\beta$ for $\beta\in \mathbb F_q^\times$ and thus can take at most $(q-1) \abs{\pi} $ possible values. There are $N+1= n \deg \pi$ possible values of $\deg h$, so it suffices to check that for each such congruence class, and each value of $\deg h$, there can be at most one pair $(g,h)$ satisfying all the conditions.

If $g_1/h_1 \equiv g_2/h_2 \bmod \pi^{n}$, $\deg h_1=\deg h_2$, and $\deg g_1+\deg h_1, \deg g_2+\deg h_2 \leq N$ then $g_1 h_2 = g_2 h_1 \bmod \pi^{n}$.  Furthermore $\deg (g_1 h_2) =\deg g_1 + \deg h_2 = \deg g_1 + \deg h_1 \leq N$ and similarly $\deg (g_2 h_2) \leq N$. Thus we  have  $g_1h_2= g_2 h_1$. Then because $\gcd(g_1,h_1)=\gcd(g_2,h_2)=1$ and all the polynomials are monic, we must have $g_1= g_2$ and $h_1= h_2$, as desired. \end{proof}

Hence the sum over $d_1+d_2\leq N$ of \eqref{first-coefficient-difference} is bounded by $n  (\deg \pi) ^2 q^{\frac{1}{2}}  (q-1) \abs{\pi}^{\frac{n}{2}+2 }   =O( n  \abs{\pi}^{\frac{n}{2}} )$.

For $d_1+d_2>N$, by \eqref{coefficient-formula}, the functional equation, and Lemma \ref{coefficient-average-evaluation}, the coefficient of $q^{-d_1 \alpha_1 + d_2 \alpha_2}$ in the left-hand side of \eqref{CFKRS-like} is 
\[ \sum_{ \chi \in  \mathcal F_{\pi,n}}\chi(a)a_{d_1} \overline{a_{d_2}} = \sum_{ \chi \in  \mathcal F_{\pi,n}}\chi(a)\overline{ a_{N- d_1}} a_{N-d_2} = q^{ \frac{d_1+d_2}{2} - N }   \sum_{ \substack{(g,h)\in \mathcal Q\\ \gcd(g,h)=1 \\  \deg g- \deg h = d_2-d_1    }}  C_{g,h}  A_{N-d_1 -\deg g}.\]
The difference between this and \eqref{RHS-coefficient} is 
\begin{equation} \sum_{ \substack{ (g,h) \in \mathcal Q\\  \deg g - \deg h = d_2 -d_1 \\ \abs{g} \abs{h}  > q^N /\abs{\pi}^2  } }  C_{g,h}   q^{ \frac{d_1+d_2}{2}-N} A_{N-d_1-\deg g}  . \end{equation}
The bound for this sum is almost identical to the $d_1 + d_2\leq N$ case. We start with 
\[ q^{ \frac{d_1+d_2}{2} - N } \abs{A_{N-d_1-\deg g}} \leq q^{ \frac{d_1+d_2}{2} - N }  q^{ N-d_1-\deg g} = q^{ \frac{d_2-d_1}{2}- \deg g} = q^{ \frac{\deg g -\deg h}{2}-\deg g} = q^{ -\frac{\deg g+\deg h}{2}}\leq \frac{\abs{\pi}}{q^{\frac{N}{2}}}.  \] 
and then observe that each pair $(g,h)$ contributes to \eqref{first-coefficient-difference} for at most $\deg \pi$ pairs $d_1,d_2$, and only if $\deg g + \deg h\leq (N-d_1) + (N-d_2) < N$, so the sum over $d_1+d_2> N$ of  \eqref{first-coefficient-difference} is bounded by $ \deg \pi q^{\frac{1}{2}}  \abs{\pi}^{\frac{n}{2}+1}$ times the number of relatively prime pairs $g,h$ with $a \equiv \beta g/h \bmod \pi^{n-1}  $ for some $\beta \in \mathbb F_q^\times$ and $q^N /\abs{\pi}^2< \abs{g}\abs{h} \leq q^N$ and thus is $O ( n \abs{\pi}^{\frac{n}{2}})$.

 \section{Function field applications}

\subsection{Application to short interval sums}

Let $\mathbb F_q$ be a finite field with $q$ elements, Recall for $g\in \mathbb F_q[T]$ that $\mathcal I_{g, (k-1)(n-2)-1}$ is the set of $f\in \mathbb F_q[T]$ such that $f-g$ has degree $<(k-1) (n-2)-1 $.

We now provide the application to short interval sums of divisor-like functions. We first relate these to Kloosterman sums:

\begin{lem}\label{identity-short-interval}  let $R = \mathbb F_q[[T^{-1}]]$, and take $\pi = T^{-1}$. Let $\psi\colon R / \pi^n R  \to \mathbb C^\times$ be defined by extracting the coefficient of $T^{1-n}$ and then applying a nontrivial additive character of $\mathbb F_q$.  

Then we have the identity
\[ \sum_{f \in \mathcal I_{g, (k-1) (n-2)-1  }}  d_k^{(n-2,\dots, n-2)}(f)   =q^{(k-1)(n-2)+1} + \frac{1}{q^k} \sum_{a \in \mathbb F_q^\times}  Kl_k( ag/ T^{(n-2)k}) .\]   \end{lem}

\begin{proof} Any polynomial, divided by $T^m$, gives an element of $R$ as long as its degree is at most $m$, and this element lies in $\pi^d R$ as long as the degree is at most $m-d$, i.e. $< m+1-d$. Since $(n-2)k + 1 - n = (k -1) (n-2)-1$, we have
\[ \sum_{f \in \mathcal I_{g, (k-1) (n-2)-1  } }d_k^{(n-2,\dots, n-2)}(f)\] \[  = \# \{   f_1,\dots, f_k \in \mathbb F_q[T] ^+\mid \deg(f_i)= n-2 ,  \deg( \prod_{i=1}^k f_i - g) < (k-1) (n-2)-1  \} \]
\[ =  \# \{   f_1,\dots, f_k \in \mathbb F_q[T]^+ \mid  \deg(f_i)= n-2 ,  \prod_{i=1}^k (f_i/T^{n-2} ) - g / T^{(n-2)k} \in \pi^n R \}\]

An element $y\in R/ \pi^n R$ has the form $f/ T^{n-2}$ for some monic $f$ of degree $n$ if and only if $y \equiv 1 \bmod \pi$ and $\psi ( a y)=1$ for all $a \in \mathbb F_q$, and $f$, if it exists, is unique.  This is because we may write $x = c_0 + c_1 T^{-1} + \dots + c_{n-1} T^{n-1}$, the first condition is equivalent to $c_0=1$, the second condition is equivalent to $c_{n-1}=0$, and then the unique $f$ that works is $c_0 T^{n-2} + c_1 T^{n-3} + \dots + c_{n-2}$. Thus
\[ \sum_{f \in \mathcal I_{g, (k-1) (n-2)-1  } }d_k^{(n-2,\dots, n-2)}(f)  \] \[=  \# \{  y_1,\dots, y_k \in R/\pi^n R \mid  y_i \equiv 1\bmod \pi, \psi(ay_i)=1 \textrm{ for all } a, \prod_{i=1}^k y_i \equiv  g / T^{(n-2)k} \bmod \pi^n R \} \]
\[=  \frac{1}{q^k} \sum_{a_1,\dots, a_k \in \mathbb F_q} \sum_{\substack{ y_1,\dots, y_k \in R/\pi^n R \\ y_i \equiv 1\bmod \pi, \\  \prod_{i=1}^k y_i \equiv  g / T^{(n-2)k} \bmod \pi^n R}} \psi ( \sum_{i=1}^k a_i y_i ).  \]

We now consider the inner sum. If all $a_i$ are zero, the inner sum is trivial, and equal to $ q^{ (k-1) (n-1) }$ as there are $q^{n-1}$ possibilities for each $y_i$ and the equation uniquely determines $y_k$ in terms of the other $y_i$. This term contributes $q^{ (k-1)(n-1) -k} = q^{ (k-1) (n-2)-1}$. If $a_j=0$ for some $j$ but not for all $j$, then as $y_j$ is uniquely determined by the equation from the other $y_i$, we can eliminate the variable, at which point the sum splits as a product $\prod_{i \neq j} \sum_{ \substack{y_i \in R/\pi^n R \\ y_i \equiv 1 \bmod \pi}} \psi (a_iy_i) $ which is zero since the factor corresponding to any $i$ with $a_i \neq 0$ vanishes. This gives
\[ \sum_{f \in \mathcal I_{g, (k-1) (n-2)-1  }} d_k^{(n-2,\dots, n-2)}(f)  =  q^{ (k-1) (n-2)-1} +  \frac{1}{q^k} \sum_{a_1,\dots, a_k \in \mathbb F_q^\times } \sum_{\substack{ y_1,\dots, y_k \in R/\pi^n R \\ y_i \equiv 1\bmod \pi, \\  \prod_{i=1}^k y_i \equiv  g / T^{(n-2)k} \bmod \pi^n R}} \psi ( \sum_{i=1}^k a_i y_i ).  \]
Now writing $x_i = a_i y_i$, using the fact that each element of $(R/\pi^n)^\times$ arises as $a_iy_i$ for a unique $a_i\in\mathbb F_q^\times$ and $y_i \in \mathbb R/\pi^n$ congruent to $1$ mod $\pi$, and $\prod_{i=1}^k x_i =\prod_{i=1}^k a_i \prod_{i=1}^k y_i = a g / T^{(n-2)k} $ for some $g\in \mathbb F_q^\times$, we obtain
\[ \sum_{f \in \mathcal I_{g, (k-1) (n-2)-1  }} d_k^{(n-2,\dots, n-2)}(f)  =  q^{ (k-1) (n-2)-1} +  \frac{1}{q^k} \sum_{a\in \mathbb F_q^\times} \sum_{\substack{ x_1,\dots, x_k \in (R/\pi^n R)^\times \\  \prod_{i=1}^k x_i \equiv  a g / T^{(n-2)k} \bmod \pi^n R}} \psi ( \sum_{i=1}^k x_i ).  \]
We recognize the inner sum as a Kloosterman sum. \end{proof}

\begin{lem}\label{nonv-short-interval} We have \[ \sum_{f \in \mathcal I_{g, (k-1) (n-2)-1  }} d_k^{(n-2,\dots, n-2)}(f)   =q^{(k-1)(n-2)+1} \] for all but at most $ q^{  \lceil \frac{n}{ p^v+1} \rceil + \lceil \frac{n-1}{ p^v+1} \rceil-1 } (q-1)$ choices of $g$ modulo polynomials of degree $< (k-1)(n-2)-1$. \end{lem}

Note that the choice of $g$  modulo polynomials of degree $< (k-1)(n-2)-1$ is the same as the choice of interval.

\begin{proof}  By \cref{identity-short-interval}, this identity holds unless $ Kl_k( ag/ T^{(n-2)k})\neq 0$ for some $a\in \mathbb F_q^\times$. Each value of $ag/ T^{(n-2)k}$ can occur for only one choice of (monic) $g$  modulo polynomials of degree $< (k-1)(n-2)-1$, so it suffices to bound the number of $x \in R/\pi^n$ for which $Kl_k(x)\neq 0$. We then apply \cref{number-of-nonzero-values}, and observe that $\abs{R/\pi}=q$, $c= \lceil \frac{n}{ p^v+1} \rceil$, and $\tilde{c} = \lceil \frac{n-1}{ p^v+1} \rceil$. \end{proof}

\begin{lem}\label{interval-lower-bound} We have \[ \Bigl| \sum_{f \in \mathcal I_{g, (k-1) (n-2)-1  } }d_k^{(n-2,\dots, n-2)}(f)  - q^{(k-1)(n-2)+1} \Bigr| \geq q^{ \frac{1}{2} \left( k (n-3) -  \lceil \frac{n}{ p^v+1} \rceil - \lceil \frac{n-1}{ p^v+1} \rceil+ 1\right)} (q-1)^{\frac{k-1}{2}} \] for at least one value of $g$. \end{lem}

\begin{proof} Let $G$ be the group $ (1 + T^{-1} \mathbb F_q[[T^{-1}])^\times / (1+ T^{-n} \mathbb F_q[[T{^-1} ]])^\times$ of elements congruent to $1$ mod $T^{-1}$ in $\mathbb F_q[[T^{-1}]] / T^{-n} \mathbb F_q[[T^{-1}]]$, whose elements may be uniquely expressed as  $ 1+ c_1 T^{-1}  + \dots + c_{n-1}T^{1-n}$ for $c_1,\dots, c_{n-1}\in \mathbb F_q$. Given such a tuple $\mathbf c$, let $x_{\mathbf c}$ be the corresponding element $ 1+ c_1 T^{-1}  + \dots + c_{n-1}T^{1-n}$, and let $T^m x_{\mathbf c}= T^m + c_1 T^{m-1} + \dots + c_{n-1} T^{m+1-n}.$ By the Plancherel formula applied to $G$, we have
  \[ \sum_{\mathbf c \in \mathbb F_q^{n-1}}  \Bigl| \sum_{f \in \mathcal I_{T^{ k (n-2)} x_{\mathbf c}  , (k-1) (n-2)-1  }} d_k^{(n-2,\dots, n-2)}(f)  - q^{(k-1)(n-2)+1} \Bigr|^2 \] \[ =\frac{1}{q^{n-1}}  \sum_{\chi  \colon G \to \mathbb C^\times}   \Bigl|\sum_{\mathbf c \in \mathbb F_q^{n-1}}  \chi( 1 + c_1 T^{n-1}  + \dots + c_{n-1}T^{1-n} )\Bigl( \sum_{f \in \mathcal I_{T^{ k (n-2)} x_{\mathbf c}  , (k-1) (n-2)-1  }}d_k^{(n-2,\dots, n-2)}(f)  - q^{(k-1)(n-2)+1}\Bigr)\Bigr|^2 \]
  \[= \frac{1}{q^{n-1}}  \sum_{\chi  \colon G \to \mathbb C^\times}   \Bigl| \sum_{\substack{f_1,\dots, f_k \in \mathbb F_q[T]^+  \\ \deg f_i =n-2}}\chi\Bigl( \prod_{i=1}^k \frac{f_i}{T^{n-2}} \Bigr) -\sum_{\mathbf c \in \mathbb F_q^{n-1}}  \chi( 1 + c_1 T^{n-1}  + \dots + c_{n-1}T^{1-n} ) q^{(k-1)(n-2)+1}\Bigr|^2. \]
For $\chi$ trivial, we have $\sum_{\substack{f_1,\dots, f_k \in \mathbb F_q[T]^+ \\ \deg f_i =n-2}}\chi\Bigl( \prod_{i=1}^k \frac{f_i}{T^{n-2}} \Bigr)= q^{ k (n-1)}$ and $\sum_{\mathbf c \in \mathbb F_q^{n-1}}  \chi( x_{\mathbf c} ) q^{(k-1)(n-2)+1}=q^{k(n-1)}$, so these terms cancel. For $\chi$ nontrivial, $\sum_{\mathbf c \in \mathbb F_q^{n-1}}  \chi(x_{\mathbf c} ) q^{(k-1)(n-2)+1}=0$. This gives
\[ \sum_{\mathbf c \in \mathbb F_q^{n-1}}  \Bigl| \sum_{f \in \mathcal I_{T^{ k (n-2)} x_{\mathbf c}  , (k-1) (n-2)-1  }} d_k^{(n-2,\dots, n-2)}(f)  - q^{(k-1)(n-2)+1} \Bigr|^2 \] 
\[ =  \frac{1}{q^{n-1}}  \sum_{\substack{\chi  \colon G \to \mathbb C^\times\\ \chi\neq 1}}   \Bigl| \sum_{\substack{f_1,\dots, f_k \in \mathbb F_q[T]^+ \\ \deg f_i =n-2}}\chi\Bigl( \prod_{i=1}^k \frac{f_i}{T^{n-2}} \Bigr) \Bigr|^2 =  \sum_{\substack{\chi  \colon G \to \mathbb C^\times\\ \chi\neq 1}}   \Bigl| \sum_{\substack{f\in \mathbb F_q[T]^+ \\ \deg f =n-2}}\chi\Bigl(\frac{f}{T^{n-2}} \Bigr) \Bigr|^{2k} \]
\[ \geq  \frac{1}{q^{n-1}}    \frac{1}{ (q^{n-1}-1)^{k-1} }\Bigl( \sum_{\substack{\chi  \colon G \to \mathbb C^\times\\ \chi\neq 1}}   \Bigl| \sum_{\substack{f\in \mathbb F_q[T]^+ \\ \deg f =n-2}}\chi\Bigl(\frac{f}{T^{n-2}} \Bigr) \Bigr|^2\Bigr)^k \]
by H\"older's inequality. Now by Plancherel again
\[ \sum_{\substack{\chi  \colon G \to \mathbb C^\times\\ \chi\neq 1}}   \Bigl| \sum_{\substack{f\in \mathbb F_q[T]^+ \\ \deg f =n-2}}\chi\Bigl(\frac{f}{T^{n-2}} \Bigr) \Bigr|^2 = \sum_{\substack{\chi  \colon G \to \mathbb C^\times}}   \Bigl| \sum_{\substack{f\in \mathbb F_q[T]^+ \\ \deg f =n-2}}\chi\Bigl(\frac{f}{T^{n-2}} \Bigr) \Bigr|^2 - q^{2(n-2)}= q^{ n-1 } \sum_{x \in G} \Bigl|  \sum_{\substack{f\in \mathbb F_q[T]^+ \\ \deg f =n-2\\ f/T^{n-2} =x}}1\Bigr|^2 - q^{2(n-2)} \] \[= q^{n-1}q^{n-2} - q^{ 2(n-2)} = (q-1) q^{2(n-2) } \]
so
  \[ \sum_{\mathbf c \in \mathbb F_q^{n-1}}  \Bigl| \sum_{f \in \mathcal I_{T^{ k (n-2)} x_{\mathbf c}  , (k-1) (n-2)-1  }} d_k^{(n-2,\dots, n-2)}(f)  - q^{(k-1)(n-2)+1} \Bigr|^2 \geq  \frac{ q^{ 2k(n-2)} (q-1)^k} { q^{n-1} (q^{n-1} -1)^{k-1} }  \geq  q^{ k (n-3)} (q-1)^k. \]
  By \cref{nonv-short-interval}, the summand can be nonvanishing for at most $ q^{  \lceil \frac{n}{ p^v+1} \rceil + \lceil \frac{n-1}{ p^v+1} \rceil-1 } (q-1)$  values of $\mathbf c$, so one value of $\mathbf c$ must contribute at least
 \[ q^{ k (n-3) -  \lceil \frac{n}{ p^v+1} \rceil - \lceil \frac{n-1}{ p^v+1} \rceil+ 1} (q-1)^{k-1}\] to the sum, meaning the error term has size at least
  \[ q^{ \frac{1}{2} \left( k (n-3) -  \lceil \frac{n}{ p^v+1} \rceil - \lceil \frac{n-1}{ p^v+1} \rceil+ 1\right)} (q-1)^{\frac{k-1}{2}}.\]

\end{proof}

\begin{proof}[Proof of \cref{interval-lower-bound-intro}] This follows from Lemma \ref{interval-lower-bound} after inputting $\lceil \frac{n}{ p^v+1}  \rceil \leq \frac{n}{ p^v+1} $ and then collecting all the terms depending only on $q,k$ into the implicit constant. \end{proof}

\subsection{Application to moments of Dirichlet $L$-functions}

Finally, we explain why the error term for \eqref{CFKRS-like} cannot admit square-root cancellation.

We note that $L(s,\chi)$ can be expressed as a polynomial in $q^{-s}$ with constant term $1$ and leading term $ \epsilon_\chi q^{ \frac{ n \deg \pi -1 }{2} } q^{ -(n\deg \pi -1) s}$, where $\epsilon_\chi$ is the constant in the functional equation of $L(s,\chi)$.  Using this polynomiality, we obtain the contour integral evaluations
\[ \frac{\log q}{2\pi i} \int_0^{\frac{2\pi i}{\log q}} L (1/2 + \alpha, \chi) d\alpha =1\]
and
\[ \frac{\log q}{2\pi i} \int_0^{\frac{2\pi i}{\log q}} q^{ (n \deg \pi-1) \alpha}  L (1/2 + \alpha, \chi) d\alpha =\epsilon_\chi\]
which together imply that, setting $v = \lfloor \log k/\log p\rfloor$,
\[ \left( \frac{\log q}{2\pi i} \right)^{2k} \int_0^{\frac{2\pi i}{\log q}}  \dots \int_0^{\frac{2\pi i}{\log q}}  q^{ \sum_{i=1}^{p^v} (n\deg \pi-1)\alpha_i} \prod_{i=1}^k L(1/2 + \alpha_i, \chi) \overline{L(1/2+ \alpha_{k+i},\chi)} d\alpha_1 \dots d \alpha_{2k} = \epsilon_\chi^{p^v}\]
so that
\begin{equation} \label{big-contour-integral} \begin{split} & \hspace{-.5in} \left( \frac{\log q}{2\pi i} \right)^{2k} \int_0^{\frac{2\pi i}{\log q}}  \dots \int_0^{\frac{2\pi i}{\log q}}  q^{ \sum_{i=1}^{p^v} (n\deg \pi-1)\alpha_i}  \sum_{ \chi \in  \mathcal F_{\pi,n}} \chi(a)\prod_{i=1}^k L(1/2 + \alpha_i, \chi) \overline{L(1/2+ \alpha_{k+i},\chi)} d\alpha_1 \dots d \alpha_{2k} \\ & =  \sum_{ \chi \in  \mathcal F_{\pi,n}} \chi(a)\epsilon_\chi^{p^v} .\end{split}\end{equation}
Assuming \eqref{CFKRS-like} with a given power savings $\delta$, we may contour integrate both sides against $q^{ \sum_{i=1}^{p^v} N \alpha_i}$ and thus obtain an estimate for \eqref{big-contour-integral}. 

Contour integrating the error term $O( \abs{\pi} ^{ (1-\delta) n} ) $ of \eqref{CFKRS-like} simply gives an error term of $O( \abs{\pi} ^{ (1-\delta) n} ) $.

Contour integrating the main term of \eqref{CFKRS-like} against $q^{ \sum_{i=1}^{p^v} N \alpha_i}$ has the effect of cancelling all terms where the coefficient of $\alpha_i$ in the exponent of $q$ is not equal to $- N$ for some $i\leq p^v$ or not equal to $ 0$ for some $i> p^v$. In particular, it cancels terms where the sum over $i$ of the coefficient of $\alpha_i$ in the exponent of $q$ is not equal to $- N  p^v$. However, using the equation $g\prod_{i\notin S} f_i =\beta_g h \prod_{i\in S} f_i$ to obtain $\deg g + \sum_{i\notin S} \deg f_i = \deg h + \sum_{i\in S} \deg f_i$ and using $\abs{S}=k$, we see that this exponent is $\deg h-\deg g$. Since $\deg g+\deg h\leq N-2 \deg \pi < N$, we have $\abs{\deg h-\deg g} < N$, so we cannot have $\deg h-\deg g = -N p^v$. Thus all the terms cancel and the contour integral vanishes.

Thus \eqref{CFKRS-like} with any power savings $\delta$ implies \eqref{big-contour-integral} is  $O( \abs{\pi} ^{ (1-\delta) n} )$.

 We now estimate the right side of \eqref{big-contour-integral} in terms of Kloosterman sums.
 
Let $R= \mathbb F_q[T]_{\pi}$ be the localization of $\mathbb F_q[T]$ at $\pi$. Let $\psi\colon \mathbb F_q[T] / \pi^n \mathbb F_q[T]  \to \mathbb C^\times$ be defined by extracting the coefficient of $T^{n \deg \pi -1}$ and then applying a nontrivial additive character of $\mathbb F_q$. 
 
 \begin{lem}\label{identity-moments} We have \[ \sum_{ \chi \in  \mathcal F_{\pi,n}} \chi(a)  \epsilon_\chi^{ p^v} =  \frac{  q^{ - \frac{ p^v(n \deg \pi +1) }{2} }}{  \abs{\pi}^{n-1} (\abs{\pi}-1)}  \sum_{\lambda_1,\dots, \lambda_{p^v} \in \mathbb F_q^\times} \psi \Bigl( \sum_{i=1}^{p^v} \lambda_i  T^{ n \deg \pi -1}\Bigr)  Kl_k \left( \frac{ \prod_{i=1}^{p^v} \lambda_i}{a} \right) .\] \end{lem}
 
\begin{proof}  We first express $\epsilon_\chi$ in terms of Gauss sums. We have \[ \epsilon_\chi = q^{ - \frac{ n \deg \pi -1 }{2} }  \sum_{ \substack{ f  \in \mathbb F_q[T]^+ \\ \deg f= n \deg \pi -1 }} \chi(f) =  q^{ - \frac{ n \deg \pi +1 }{2} }   \sum_{\lambda \in \mathbb F_q} \psi(-\lambda T^{n \deg \pi -1}  ) \sum_{ \substack{ f  \in \mathbb F_q[T]/ \pi^n }} \chi(f) \psi(\lambda f) \] \[ =   q^{ - \frac{ n \deg \pi +1 }{2} }   \sum_{\lambda \in \mathbb F_q^\times } \psi(-\lambda T^{n \deg \pi -1}  ) \sum_{ \substack{ f  \in \mathbb F_q[T]/ \pi^n }} \chi(f) \psi(\lambda f) \] \[= q^{ - \frac{ n \deg \pi +1 }{2} }    \sum_{\lambda \in \mathbb F_q^\times } \psi(-\lambda T^{n \deg \pi -1}  ) \chi(\lambda^{-1} )   \sum_{ \substack{ f  \in \mathbb F_q[T]/ \pi^n }} \chi(f) \psi( f) . \] 
 
 Thus \[ \sum_{ \chi \in  \mathcal F_{\pi,n}} \chi(a)  \epsilon_\chi^{ p^v} \] \[= q^{ - \frac{ p^v(n \deg \pi +1) }{2} }   \sum_{ \chi \in  \mathcal F_{\pi,n}} \chi(a) \Bigl(\sum_{\lambda \in \mathbb F_q^\times } \psi(-\lambda T^{n \deg \pi -1}  ) \chi(\lambda^{-1} )  \Bigr)^{p^v} \Bigl( \sum_{ \substack{ f  \in \mathbb F_q[T]/ \pi^n }} \chi(f) \psi( f)  \Bigr)^{p^v} \]
  \[= q^{ - \frac{ p^v(n \deg \pi +1) }{2} } \sum_{ \substack{ \chi\colon (\mathbb F_q[T] / \pi^n)^\times \to \mathbb C^\times  }}\chi(a)  \Bigl(\sum_{\lambda \in \mathbb F_q^\times } \psi(-\lambda T^{n \deg \pi -1}  ) \chi(\lambda^{-1} )  \Bigr)^{p^v} \Bigl( \sum_{ \substack{ f  \in \mathbb F_q[T]/ \pi^n }} \chi(f) \psi( f)  \Bigr)^{p^v} \]
  \[ = \frac{  q^{ - \frac{ p^v(n \deg \pi +1) }{2} }}{  \abs{\pi}^{n-1} (\abs{\pi}-1)}  \sum_{\lambda_1,\dots, \lambda_{p^v} \in \mathbb F_q^\times } \psi \Bigl( \sum_{i=1}^{p^v} \lambda_i  T^{n \deg \pi -1}  \Bigr) \sum_{ \substack{f_1,\dots, f_{p^v} \in \mathbb F_q[T]/ \pi^n \\  a\prod_{i=1}^{p^v} f_i = \prod_{i=1}^{p^v} \lambda_i }} \psi \Bigl (\sum_{i=1}^{p^v} f_i \Bigr) \]
  \[ =   \frac{  q^{ - \frac{ p^v(n \deg \pi +1) }{2} }}{  \abs{\pi}^{n-1} (\abs{\pi}-1)}  \sum_{\lambda_1,\dots, \lambda_{p^v} \in \mathbb F_q^\times }\psi \Bigl( \sum_{i=1}^{p^v} \lambda_i  T^{n \deg \pi -1}  \Bigr)  Kl_k \left( \frac{ \prod_{i=1}^{p^v} \lambda_i}{a} \right) ,\]
  since $\sum_{\lambda \in \mathbb F_q^\times } \psi(-\lambda T^{n \deg \pi -1}  ) \chi(\lambda^{-1} )$ vanishes for $\chi$ even and $ \sum_{ \substack{ f  \in \mathbb F_q[T]/ \pi^n }} \chi(f) \psi( f) $ vanishes for $\chi$ imprimitive. \end{proof}
 
 \begin{lem}\label{nonv-moments} The moment $\sum_{ \chi \in  \mathcal F_{\pi,n}} \chi(a)  \epsilon_\chi^{ p^v} =$ is nonvanishing for at most \[ \abs{\pi} ^{  \lceil \frac{n}{ p^v+1} \rceil + \lceil \frac{n-1}{ p^v+1} \rceil-1 } (q\abs{\pi}-1)(q-1) \] choices of $a \in (\mathbb F_q[T]/\pi^n)^\times  $. \end{lem}

\begin{proof}  By \cref{identity-moments}, if the moment is nonvanishing, then $ Kl_k( \lambda/a )\neq 0$ for some $\lambda \in \mathbb F_q^\times$. Each value of $\lambda/a$ can occur for exactly $q-1$ choices of $a$, so it suffices to bound the number of $x \in R/\pi^n$ for which $Kl_k(x)\neq 0$ and then multiply by $q-1$. We then apply \cref{number-of-nonzero-values}, and observe that $\abs{R/\pi}=\abs{\pi}$, $c= \lceil \frac{n}{ p^v+1} \rceil$, and $\tilde{c} = \lceil \frac{n-1}{ p^v+1} \rceil$. \end{proof}

\begin{lem} \label{epsilon-average-lb}There exists $a \in( \mathbb F_q[T]/\pi^n)^\times$ such that  \[ \Bigl | \sum_{ \chi \in  \mathcal F_{\pi,n}} \chi(a)  \epsilon_\chi^{ p^v } \Bigr| \geq  \abs{\pi} ^{\left( 1- \frac{1}{ p^v+1} \right)  n } C \] where $C$ is a constant depending only on $q,\deg \pi, v$ and not on $n$.
\end{lem}

Since the trivial bound is the length of the sum $\abs{\pi}^n$, because the individual terms are bounded by $1$, this represents a power savings of only $\frac{1}{p^v+1}$.

\begin{proof} We have

\[ \sum_{a \in (\mathbb F_q[T]/\pi^n)^\times}  \Bigl| \sum_{ \chi \in  \mathcal F_{\pi,n}} \chi(a)  \epsilon_\chi^{ p^v} \Bigr|^2 = \abs{\pi}^{n-1}  (\abs{\pi} -1 )  \sum_{ \chi \in  \mathcal F_{\pi,n}} \abs{\epsilon_\chi}^{2 p^v}  \] \[= \abs{\pi}^{n-1}  (\abs{\pi} -1 )   \sum_{ \chi \in  \mathcal F_{\pi,n}} 1= \abs{\pi}^{n-1}  (\abs{\pi} -1 ) \cdot \abs{\pi}^{n-2}  (\abs{\pi}-1) (\abs{\pi}-2) .\]

By \cref{nonv-moments}, the number of nonvanishing terms of the sum over $a$ is at most $ \abs{\pi} ^{   \lceil \frac{n}{ p^v+1} \rceil + \lceil \frac{n-1}{ p^v+1} \rceil-1 } (\abs{\pi}-1)(q-1) $, so one of the terms must be at least \[\abs{\pi}^ {2n-2 -   \lceil \frac{n}{ p^v+1} \rceil - \lceil \frac{n-1}{ p^v+1} \rceil)}  (\abs{\pi}-1 ) (q\abs{\pi}-2)  (q-1)^{-1} .\] 

Hence one of the values of $\sum_{ \chi \in  \mathcal F_{\pi,n}} \chi(a)  \epsilon_\chi^{ p^v}$ must be at least 
\[ \abs{\pi}^{\frac{1}{2} ( 2n-2 -   \lceil \frac{n}{ p^v+1} \rceil - \lceil \frac{n-1}{ p^v+1} \rceil )}  \sqrt{ (\abs{\pi}-1 ) (\abs{\pi}-2)  (q-1)^{-1}}  \geq \abs{\pi}^{\left( 1- \frac{1}{ p^v+1} \right)  n } C\] where $C$ is a constant depending only on $q, \deg \pi, v$.
\end{proof} 

In particular, \eqref{CFKRS-like} cannot hold with $\delta> \frac{1}{p^v+1}$.

One could try to recover square-root cancellation by replacing $\epsilon$-factors by their average without taking the limit as $n\to\infty$, in which case the averages would give these Kloosterman sums. In particular, if the nonvanishing Kloosterman sums were supported on a ``diagonal set" that has a description independent of $\pi^n$, and given by a simple formula on that set, one could use this to extract a (conjectural) secondary main term. However, it does not seem that the set where $Kl_k(x)\neq 0$ admits such a nice description.

\bibliographystyle{plain}

\bibliography{references}

\end{document}